\documentclass[11pt]{amsart}
\usepackage{amssymb, graphicx, color}
\usepackage{amsfonts}
\usepackage{amsthm}
\usepackage{enumerate}
\usepackage[mathscr]{eucal}
\usepackage{graphicx}

\usepackage[pdftex, bookmarksnumbered, bookmarksopen, colorlinks, citecolor=blue, linkcolor=blue]{hyperref}

\allowdisplaybreaks

\newtheorem{thm}{Theorem}[section]
\newtheorem{defi}[thm]{Definition}

\newtheorem{lem}[thm]{Lemma}

\newtheorem{prop}[thm]{Proposition}

\newtheorem{remark}[thm]{Remark}
\newtheorem{rk}[thm]{Remark}

\theoremstyle{definition}
\theoremstyle{remark}
\numberwithin{equation}{section}

\newcommand{\R}{{\mathbb R}}
\newcommand{\N}{{\mathbb N}}
\newcommand{\Z}{{\mathbb Z}}
\newcommand{\C}{{\mathbb C}}

\newcommand{\M}{{\mathcal M}}

\newcommand{\bs}{\begin{split}}
\newcommand{\es}{\end{split}}

\theoremstyle{plain}
\newtheorem*{Cuculescutheo}{Cuculescu's construction \cite{Cuc}}

\newcommand{\be}{\begin{eqnarray*}}
\newcommand{\ee}{\end{eqnarray*}}
\newcommand{\beq}{\begin{align}}
\newcommand{\eeq}{\end{align}}

\def\Q{\mathcal{Q}}
\def\1{\mathbf{1}}
\def\lc{\lesssim}

\def\tet{\theta}

\begin{document}

%
%
%
%
%
%
%
%
\setcounter{page}{1}
\title[Maximal singular integral operators]
{Maximal singular integral operators acting on  noncommutative $L_p$-spaces}



\author[G. Hong]{Guixiang Hong$^{*}$}
\address{
School of Mathematics and Statistics\\
Wuhan University\\
Wuhan 430072\\
China}

\email{guixiang.hong@whu.edu.cn}
\author[X. Lai]{Xudong Lai}
\address{
Institute for Advanced Study in Mathematics, Harbin Institute of Technology \\
Harbin 150001\\
China}

\email{xudonglai@hit.edu.cn\ xudonglai@mail.bnu.edu.cn}
\author[B. Xu]{Bang Xu}
\address{
School of Mathematics and Statistics\\
Wuhan University\\
Wuhan 430072\\
China}

\email{bangxu@whu.edu.cn}

\thanks{G. Hong and B. Xu were supported by National Natural Science Foundation of China (No. 11601396). X. Lai was supported by China Postdoctoral Science Foundation (Nos. 2017M621253, 2018T110279), National Natural Science Foundation of China (No. 11801118) and  Fundamental Research Funds for the Central Universities.}

\subjclass[2010]{Primary 46L52; Secondary 42B20,  46L53}
\keywords{Noncommutative Calder{\'o}n-Zygmund decomposition,  Vector-valued noncommutative $L_{p}$-spaces, Maximal inequalities, Singular integral, Weak $(1,1)$}

\date{\today
\newline \indent $^{*}$Corresponding author}
\begin{abstract}
In this paper, we study the boundedness theory for maximal Calder\'on-Zygmund operators acting on noncommutative $L_p$-spaces. Our first result is a criterion for the weak type $(1,1)$ estimate of noncommutative maximal Calder\'on-Zygmund operators; as an application, we obtain the weak type $(1,1)$ estimates of operator-valued maximal singular integrals of convolution type under proper {regularity} conditions. These are the {\it first} noncommutative maximal inequalities for families of truly non-positive linear operators.
 For homogeneous singular integrals, the strong type $(p,p)$ ($1<p<\infty$) maximal estimates are shown to be true even for {rough} kernels.

 As a byproduct of the criterion, we obtain the noncommutative weak type $(1,1)$ estimate for Calder\'on-Zygmund operators with integral regularity condition that is slightly stronger than the H\"ormander condition; this evidences somewhat an affirmative answer to an open question in the noncommutative Calder\'on-Zygmund theory.

\end{abstract}

\maketitle

\section{Introduction and main results}

Motivated by quantum mechanics, operator algebra, noncommutative geometry and quantum probability, noncommutative harmonic analysis gains rapid development recently and there are many fundamental works appeared (see e.g. \cite{JMX, M, Par07, CXY13, XXY, XXX, JMP1, JMP2, GJP1, GJP2, Junge-Mei1,Junge-Mei2, Junge-Mei-Parcet-Xia, de la salle-Parcet-Ricard, HWW}). Due to the noncommutativity, many real variable tools or methods such as maximal functions, stopping times \textit{etc} are not available, which impose numerous difficulties in developing noncommutative theory. The semi-commutative (or operator-valued) harmonic analysis seems to be the \textit{easiest} noncommutative theory, but it requires revolutionary ideas, insights or techniques. Moreover, together with various transference techniques (see e.g. \cite{Neuwirth-Ricard, CXY13, Junge-Mei-Parcet-Xia, de la salle-Parcet-Ricard, HLW} and references therein), it has many exciting applications in other research fields or plays important roles in the more sophisticated noncommutative setting where the explicit expressions or the estimates of the kernels are absent.
For instance, motivated by the theory of noncommutative martingales \cite{PX1, JX1, JX2}, Mei \cite{M} developed systematically the theory of operator-valued Hardy spaces and BMO spaces which incidentally solved an open question in matrix-valued harmonic analysis arising from prediction theory. Mei's theory is also used to develop harmonic analysis on group von Neumann algebras \cite{JMP1, Junge-Mei-Parcet-Xia, de la salle-Parcet-Ricard} and quantum tori (or quantum Euclidean space) \cite{CXY13, XXY, XXX, GJP2}. Motivated by Cuculescu's maximal weak (1,1) estimate for noncommutative martingales \cite{Cuc}, Parcet \cite{JP1} formulated a kind of noncommutative Calder\'on-Zygmund decomposition and established the weak type (1,1) estimate for the operator-valued Calder\'on-Zygmund singular integrals, which finds its unexpected application in the complete resolution of the Nazarov-Peller conjecture arising from the perturbation theory \cite{CPSZ}. For more related results on weak type $(1,1)$ estimates and noncommutative Calder\'on-Zygmund decomposition we refer the reader to \cite{C,HLMP,HX,Lai,MP}.

\medskip

In the present paper, we study  the semi-commutative Calder\'on-Zygmund theory but focus on the {\it maximal} singular integral operators. In the commutative case, the maximal function or operator $Mf = \sup_n |T_nf|$ for a sequence of linear operators $(T_n)_n$ acting on some function $f$  is an important instrumental tool in the real variable theory, and is the main tool to obtain the pointwise convergence result. But it usually requires much more ideas or tools in estimating maximal inequalities $Mf$ (resp. pointwise convergence) than estimating inequalities of $T_nf$ uniformly (resp. norm convergence). For instance, the norm convergence of the Dirichlet series is equivalent to the boundedness of Hilbert transform; but the corresponding pointwise convergence is guaranteed by the famous Carleson's maximal theorem which was obtained around 40 years later. In the noncommutative setting, since the maximal function is not available any more (see \cite{JX1}), the maximal inequality is much more difficult to get. The first two non-trivial maximal inequalities go back 70's in the last century, that are,  Yeadon's maximal weak type (1,1) estimate for ergodic averages (see \cite{Ye}) and Cuculescu's one for conditional expectations mentioned previously. However, it took around 30 years to obtain the noncommutative $L_p$-maximal inequalities along the line of Cuculescu and Yeadon's theorems; the formulation of $L_p$-maximal inequalities was not possible until the appearance of Pisier's vector-valued noncommutative $L_p$-spaces (see \cite{pisier}) which required the full strength of the operator space theory. Indeed, the $L_p(\Omega)$-norm of the maximal function $Mf$ must be understood as the $L_p(\Omega;\ell_\infty)$-norm of the sequence $(T_nf)_n$ in the noncommutative case. Motivated by Pisier's definition of $\ell_\infty$-valued noncommutative $L_p$-spaces for hyperfinite algebra and the noncommutative martingale inequalities in the seminal paper  \cite{PX1}, Junge  \cite{junge} in 2002 extended Pisier's definition for general algebras and Doob's $L_p$-maximal inequality for noncommutative martingales with argument based on Hilbert module theory. A few years later, Junge and Xu  \cite{JX1} obtained $L_p$-maximal inequalities for ergodic averages; the key tool {was} a noncommutative analogue of Marcinkiewicz interpolation theorem for families of positive maps which allows to deduce results from Cuculescu and Yeadon's inequalities.

Even though the noncommutative maximal inequalities for ergodic averages and conditional expectations have now been established, there appear a lot of difficulties to obtain maximal inequalities for other families of linear operators. For instance, Mei \cite{M} has to invent an ingenious idea to show the noncommutative Hardy-Littlewood maximal inequalities based on Junge's noncommutative Doob's inequalities; in \cite{HLW}, Hong \textit{et al} exploited the group structure or probabilistic method to show the noncommutative maximal ergodic inequalities associated with the action of groups of polynomial growth; since the lack of estimates of Fourier multipliers on group algebras, the first author and his collaborators \cite{HWW} invented some quantum semigroup and analyze carefully its difference with Fourier multiplier to establish the maximal inequalities and show the pointwise convergence of noncommutative Fourier series. We refer the reader to the above mentioned papers and references therein for more results on noncommutative maximal inequalities.

As far as we know,  except some non-sharp results for Bochner-Riesz means in \cite{CXY13, HWW}, there does not exist in the literature any other non-trivial noncommutative maximal inequalities for families of non-positive linear operators, such as the truncated Calder\'on-Zygmund operators and Dirichlet means.
\medskip

In this paper, we will establish the noncommutative maximal inequalities for families of truncated operator-valued singular integrals.

\medskip
Recall that a (standard) Calder\'on-Zygmund operator (abbrieviated as CZO) $T$ is a singular integral operator in $\mathbb R^d$ mapping test functions to distributions associated with a (standard) Calder\'on-Zygmund kernel
$k: \R^d \times \R^d \setminus \{(x,x):\;x\in\mathbb R^d\} \to \C$ in the sense that
\begin{equation}\label{30}
Tf(x) = \int_{\R^d} k(x,y) f(y) \, dy,
\end{equation}
whenever $f\in C^\infty_c(\mathbb R^d)$ a test function and $x \notin \mbox{supp} \hskip1pt f$, which admits a bounded extension on $L_2(\mathbb R^d)$. A Calder\'on-Zygmund kernel $k$ is called {\emph{standard}}
if it satisfies the size condition
\begin{align}\label{1}
|k(x,y)| \ \leq \
\frac{C}{|x-y|^d}
\end{align}
for $x,y \in \R^d$, where and in the sequel, $C$ is a positive numerical constant, and the $\gamma$-Lipchitz regularity condition with $\gamma\in(0,1]$,
\begin{align}\label{232}
\begin{array}{rcl} \big| k(x,y) - k(x,z) \big|+\big| k(y,x) - k(z,x) \big|& \leq &
\displaystyle \frac{C|y-z|^\gamma}{|x-y|^{d+\gamma}}   \hskip1pt
\end{array}
\end{align}
for $x,y,z\in\R^d$ satisfying $|x-y|\geq2|y-z|$. The distance $|x-y|$ between $x$ and $y$ is taken with respect to $\ell_{2}$ metric throughout the paper. The simplest example of CZO is the Hilbert transform (or Riesz transform). It is well-known that a standard Calder\'on-Zygmund operator $T$ admits a bounded extension and thus is well-defined on $L_p$ for all $1\leq p<\infty$; moreover, the same conclusions hold still true if the $\gamma$-Lipchitz condition is weakened to some $L_{q}$-integral regularity condition with $q\in(0,+\infty)$ (see for instance \cite{HH})
\begin{align}\label{6}
\sum_{m=1}^{\infty}\delta_{q}(m)<\infty,\,
\end{align}
where
\begin{align}\label{5}
\delta_{q}(m)=\sup_{\begin{subarray}{c} R>0, y\in\R^d \\
|v|\leq R \end{subarray}}\big((2^{m}R)^{d(q-1)}\int_{2^{m}R\leq|x-y|\leq2^{m+1}R}\big| k(x,y+v) - k(x,y) \big|^{q}dx\big)^{\frac{1}{q}}.
\end{align}

If $k$ satisfies the Lipschitz smoothness condition, then $\delta_{q}(m)\leq C_{d}2^{-m\gamma}$ ($C_{d}$ is a constant depending on dimension $d$) for any $q>0$, which has nice decay property;
in view of (\ref{6}), the H\"ormander condition
\begin{align*}\label{hor}
\int_{|x-y|\geq 2|y-z|} \big| k(x,y) - k(x,z) \big|dx%
<\infty,\;\forall y,z\in\mathbb R^d   \hskip1pt
\end{align*}
can be restated as $L_{1}$-integral regularity condition $\sum_{m=1}^{\infty}\delta_{1}(m)<\infty$.  In this paper, our results will be concerned with the case $q=2$ which is slightly stronger than the H\"ormander condition. Indeed, by the H\"{o}lder inequality, $\delta_{1}(m)\leq C_{d}\delta_{2}(m)$.

\medskip

 Let $\M$ be a von Neumann algebra equipped with a normal semi-finite
faithful trace (abbrieviated as \emph{n.s.f}) $\tau$ and $\mathcal N=L_{\infty}(\R^{d})\overline{\otimes}\M$ be the von Neumann algebra tensor product equipped with tensor trace $\varphi=\int\otimes\tau$. Then we can define the associated noncommutative $L_p$-spaces $L_p(\M)$ and $L_p(\mathcal N)$. The latter can be identified as the space of $L_p(\mathcal M)$-valued $p$-th integrable functions on $\R^d$. From semi-commutative Calder\'on-Zygmund theory \cite{M, JP1}, any Calder\'on-Zygmund operator $T$ with kernel satisfying the size and the H\"ormander conditions admits completely bounded extension and thus is well-defined on $L_p(\mathcal N)$ for all $1<p<\infty$; moreover if the H\"ormander condition is strengthened to the $\gamma$-Lipchitz condition, Parcet \cite{JP1} showed that the extension $T\otimes id_\M$ is of weak type $(1,1)$ and thus well-defined on $L_1(\R^d;L_1(\M))$ (see also \cite{C}) and {\it left the sufficiency of H\"ormander condition as an open question}. For simplifying the notation, the extension $T\otimes id_\M$ will be still denoted as $T$, admitting the following expression
\begin{equation}\label{3}
Tf(x) = \int_{\R^d} k(x,y) f(y) \, dy,
\end{equation}
whenever $f$ is a $L_1(\M)\cap L_\infty(\M)$-valued compactly supported measurable function and $x \notin\overrightarrow{\mathrm{supp}}  \hskip1pt f$ which is the support of $f$ as an
operator-valued function in $\R^d$.

Given a Calder\'on-Zygmund operator $T$ with kernel $k$. For any $\varepsilon>0$, we define the associated truncated singular integrals $T_{\varepsilon}f$ by
\begin{equation}\label{maximal}
T_{\varepsilon}f(x) =\int_{|x-y|>\varepsilon} k(x,y) f(y)dy.
\end{equation}
It is well-known that $T_\varepsilon$'s are Calder\'on-Zygmund operators with kernels satisfying the same conditions as those of $k$.

Our first result is a criterion for the noncommutative weak type $(1,1)$ estimate of  maximal Calder\'on-Zygmund operators; the reader is referred to Section 2 for the notions of $\ell_\infty$-valued noncommutative $L_{p}$ spaces $L_{p}(\mathcal N;\ell_{\infty})$.

\begin{thm}\label{p2}\rm
Let $T$ be a Calder\'on-Zygmund operator defined as (\ref{3}) associated to a kernel satisfying (\ref{1}) and (\ref{6}) with $q=2$ and let $T_{\varepsilon}$ be defined as (\ref{maximal}). Assume that there exists one $p_0\in (1,\infty)$ such that $(T_{\varepsilon})_{\varepsilon>0}$ is of strong type $(p_0,p_0)$, that is,
\begin{align}\label{p condition}
\|(T_{\varepsilon}f)_{\varepsilon>0}\|_{L_{p_{0}}(\mathcal N;\ell_{\infty})}\leq C\|f\|_{p_0},\;\ \ \forall f\in L_{p_{0}}(\mathcal N).
\end{align}
Then $(T_{\varepsilon})_{\varepsilon>0}$ is of weak type $(1,1)$, that is,
 for any $f\in L_{1}(\mathcal N)$ and $\lambda>0$, there exists a projection $e\in\mathcal{N}$ such that
\begin{align}\label{1 conclu}
\sup_{\varepsilon>0}\big\|e (T_{\varepsilon}f)e\big\|_\infty\leq\lambda\quad\mbox{and}\quad \varphi (e^{\perp})\leq C_{d}\lambda^{-1}\|f\|_1.
\end{align}
\end{thm}

One is able to formulate the weak type $(1,1)$ estimate \eqref{1 conclu} in a similar way as the strong type \eqref{p condition} if using the $\ell_\infty$-valued weak $L_p$ space $\Lambda_{p,\infty}(\mathcal N;\ell_\infty)$ (for its definition see \eqref{weakm}).

In the process of showing Theorem \ref{p2}, we observe that the argument essentially work also for Calder\'on-Zygmund operator itself with $L_{2}$-integral regularity condition. As mentioned previously, this solves partially a conjecture (see e.g. \cite[Page 575]{JP1}, that is, the pointwise $\gamma$-Lipchitz condition can be weakened to some integral regularity condition. We will leave the details of the proof of the following result to the interested readers and we refer the reader to Section 2 for the definition of weak $L_{p}$-space $L_{p,\infty}(\mathcal N)$.
\begin{thm}\label{t2}\rm
Let $T$ be defined as (\ref{3}) associated to a kernel satisfying (\ref{1}) and (\ref{6}) with $q=2$.
Assume
that $T$ is bounded on $L_{p_{0}}(\mathcal N)$ for some $p_{0}$ with $1<p_{0}<\infty$. Then  for any $f\in L_1(\mathcal N)$,
$$\|Tf\|_{L_{1,\infty}(\mathcal N)}\leq C_{d}\|f\|_{1}.$$
\end{thm}


In the commutative case, that is $\M=\mathbb C$, under the assumption that the kernel satisfies the $\gamma$-Lipchitz regularity condition, conclusion \eqref{1 conclu}, as well as assumption \eqref{p condition}, can be deduced from the following Cotlar inequality: for all $0<s<1$,
\begin{align}\label{weak cotlar}
\sup_{\varepsilon>0}|T_\varepsilon f(x)|\leq C_{d,\gamma}\big[\sup_{\varepsilon>0}M_\varepsilon (|f|)(x)+\sup_{\varepsilon>0}(M_\varepsilon(|Tf|^s)(x))^{\frac1s}\big],
\end{align}
where $M_\varepsilon$ is the Hardy-Littlewood averaging operator
\begin{equation}\label{e:20Mr}
M_\varepsilon g(x)=\frac{1}{\varepsilon^{d}}\int_{|x-y|\leq \varepsilon}g(y)dy,
\end{equation}
and $C_{d,\gamma}$ is a constant depending on $d$ and $\gamma$. The Cotlar inequality \eqref{weak cotlar} in turn follows from some pointwise localized argument (see e.g. \cite[Theorem 4.2.4]{Gra12008}). But the above pointwise estimate \eqref{weak cotlar} and its proof seem impossible to admit noncommutative analogues. On the other hand, in the commutative case, the $L_{2}$-integral regularity condition seems not sufficient for a weak Cotlar inequality, that is, \eqref{weak cotlar} with $s=1$, which would be enough for strong type $(p,p)$ estimate for maximal CZOs for all $1<p<\infty$. In conclusion, under the conditions (\ref{1}) and (\ref{6}) with $q=2$,  it is not clear at all whether \eqref{p condition} holds for any $p_0\in (1,\infty)$.



However, when the CZO is of convolution type, that is, $k(x,y)=k(x-y)$, a noncommutative variant of the weak Cotlar inequality---\eqref{weak cotlar} with $s=1$---in terms of norms can be  verified under the $\gamma$-Lipchitz regularity condition (see \eqref{cotlar norm}). This might be known to experts, but we will formulate it rigorously.

 It is well-known (see e.g. \cite{Gra2008}) that under the size condition \eqref{1} and the $\gamma$-Lipschitz condtion \eqref{232}, $T$ is a standard CZO of convolution type if and only if its associated kernel $k$ satisfies additionally the cancellation condition
\begin{align}\label{con4}
\sup_{0<r<R<\infty}\Big|\int_{r<|x|<R}k(x)dx\Big|<\infty.
\end{align}
It is also known from classical Calder\'on-Zygmund theory \cite{Gra2008} that \eqref{con4} implies that there exists a subsequence $(\varepsilon_j)_{j\in\mathbb N}\subset(0,+\infty)$  with $\varepsilon_j\to 0$ as $j\to+\infty$ such that
\begin{align}\label{can}
\lim_{j\rightarrow+\infty}\int_{\varepsilon_j<|x|\leq1}k(x)dx \;\mathrm{exists}
\end{align}
and for any $f\in L_{p}(\mathbb R^d)$ with $1\leq p<\infty$,
\begin{align}\label{a.e.}
T_{\varepsilon_j}f(x)\rightarrow Tf(x) \;\mathrm{a.e.},\;\mathrm{as}\;j\to+\infty.
\end{align}


We state our second result of this paper.

\begin{thm}\label{t1}\rm
Let $T$ be a CZO of convolution type with  associated kernel satisfying (\ref{1}), (\ref{232}) and (\ref{con4}). Then we have the following conclusions.
\begin{itemize}
\item[(i)] For $1<p<\infty$, $(T_{\varepsilon})_{\varepsilon>0}$ is of strong type $(p,p)$.

\item[(ii)]$(T_{\varepsilon})_{\varepsilon>0}$ is of weak type $(1,1)$.

\item[(iii)] For any $f\in L_{p}(\mathcal N)$ with $1\leq p<\infty$,
$$T_{\varepsilon_j} f\xrightarrow{\rm b.a.u} Tf \ \ \text{as} \ \ j\to\ +\infty,$$
where $(\varepsilon_j)_{j\in\N}$ is the subsequence appeared in (\ref{can}).
In other words, the principle value $Tf$ exists in the sense of b.a.u. for any $f\in L_{p}(\mathcal N)$ with $1\leq p<\infty$.
\end{itemize}
\end{thm}

Here b.a.u. denotes the noncommutative analogue of the notion of almost everywhere convergence and we refer the reader to Definition \ref{3456} in the last section for information.

\medskip

As mentioned previously, even in the commutative case, it is unclear whether Theorem \ref{t1}~(i) holds if the $\gamma$-Lipschitz condition \eqref{232} is weakened to some integral regularity condition \eqref{6}. However, if the kernel has further homogeneous property, that is, has the form ${\Omega({x})}/{|x|^{d}}$ where $\Omega$ is a homogeneous function defined on $\R^{d}\backslash \{0\}$ with degree zero
\begin{align}\label{44899}
\Omega(\lambda x')=\Omega(x'),\;\forall \lambda>0\; \mathrm{and} \;x'\in S^{d-1}
\end{align}
and integrable on the sphere $S^{d-1}$ with mean value zero
$$\int_{S^{d-1}}\Omega(x')d\sigma(x')=0,$$
then any regularity condition is  not necessary for all the strong type $(p,p)$ estimates of $(T_{\Omega,\varepsilon})_{\varepsilon>0}$, where
 \begin{align}\label{max32}
T_{\Omega,\varepsilon}f(x)=\int_{|x-y|>\varepsilon}
\frac{\Omega({x-y})}{|x-y|^{d}}f(y)dy.
\end{align}
Instead, we only need some integrability condition.
For $s\geq0$, we say that $\Omega\in L(\log^{+}L)^{s}(S^{d-1})$ if and only if
$$\int_{S^{d-1}}|\Omega(\theta)|[\log(2+|\Omega(\theta)|)]^{s}
d\theta<\infty,$$
where $d\theta=d\sigma(\theta)$ denotes the sphere measure of $S^{d-1}$. When $s=1$, we use the convention $L\log^{+}L(S^{d-1})\triangleq L(\log^{+}L)^{1}(S^{d-1})$.  
It is well-known that the following inclusion relations hold:
$$L_{2}(S^{d-1})\subset L(\log^{+}L)^{2}(S^{d-1})\subset L\log^{+}L(S^{d-1})\subset L_{1}(S^{d-1}).$$

For the endpoint case $p=1$, even in the commutative setting, it is still a conjecture that the maximal homogeneous singular integral operator with rough kernel is of weak type $(1,1)$. Thus at the moment, we are content with getting results by imposing
the so-called $L_{2}$-Dini assumption on $\Omega$
\begin{align}\label{max33}
\int_{0}^{1}\frac{\omega_{2}(s)}{s}ds<\infty,
\end{align}
where
$$\omega_{2}(\delta)=\sup_{0<|\alpha|<\delta}\Big(\int_{S^{d-1}}|\Omega(\theta)
-\Omega(\theta+\alpha)
|^{2}d\theta\Big)^{\frac{1}{2}}.$$
Now we state the third result of this paper.
\begin{thm}\label{t3}\rm
Let $T_{\Omega}$ be a homogeneous singular integral operator with $\Omega$ integrable on $S^{d-1}$ with mean value zero. Then we have the following conclusions.
\begin{itemize}
\item[(i)] Let $\Omega\in L\log^{+}L(S^{d-1})$. Then $(T_{\Omega,\varepsilon})_{\varepsilon>0}$ is of strong type $(p,p)$ with $1<p<\infty$.

\item[(ii)] Let $\Omega\in L_{2}(S^{d-1})$ satisfy (\ref{max33}). Then $(T_{\Omega,\varepsilon})_{\varepsilon>0}$ is of weak type $(1,1)$.

\item[(iii)] Let $\Omega\in L\log^{+}L(S^{d-1})$. Then for any
 $f\in L_{p}(\mathcal N)$ with $1<p<\infty$
$$T_{\Omega,\varepsilon} f\xrightarrow{\rm b.a.u} T_\Omega f \ \ \text{as} \ \ \varepsilon\to\ 0.$$
Moreover if  $\Omega\in L_{2}(S^{d-1})$ satisfies (\ref{max33}), then the above b.a.u. convergence holds for $f\in L_1(\mathcal N)$.
\end{itemize}
\end{thm}

\medskip

Theorem \ref{t1} and \ref{t3}, as Mei's and Parcet's aforementioned results in the first paragraph, together with the transference techniques (see e.g. \cite{HLW}), have applications in other related topics such as ergodic theory. Indeed, given a trace-preserving automorphic action $\alpha$ of $\mathbb R^d$ on $\mathcal M$, then $\alpha$ extends to an isometric automorphism on $L_p(\M)$ for all $0<p<\infty$. Let $f\in L_{p}(\M)$ with $1\leq p<\infty$. Let $k$ be a complex-valued measurable function defined on $\mathbb R^d\setminus\{0\}$ satisfying the assumptions in Theorem \ref{t3} (depending on the case $p=1$ or $p>1$), the result in \cite{HLW} implies that the element of the operator $T_k$ induced by $k$ acting on $f\in L_p(\M)$ exists as a principle value in the sense of b.a.u. convergence, that is,
$$T_kf=\lim_{\varepsilon\to0}\int_{|x|>\varepsilon}k(x)\alpha_xfdx,\;b.a.u.$$
If $k$ satisfies the assumptions in Theorem \ref{t1}, then the above limit has to be taken along $(\varepsilon_j)_j$.

\medskip

Let us briefly describe the strategy of the proof of Theorem \ref{p2}.  As mentioned after Theorem \ref{t2}, the commutative argument based on the Cotlar inequality \eqref{weak cotlar} do not work in the noncommutative setting.
Thus we have to provide a different approach. We start with two reductions. The first one is to reduce a general CZO to two selfadjoint ones (see Lemma \ref{mainlemma2}); this reduction is not difficult but essential in our argument in dealing with noncommutative maximal estimates. The second one is to reduce the desired maximal estimate of $(T_\varepsilon)_{\varepsilon>0}$ to that of the lacunary subsequence $(T^{\phi}_{2^j})_{j\in\mathbb Z}$ (see (\ref{0090})) via the noncommutative Hardy-Littlewood maximal inequalities; this reduction plays a key role in the present noncommutative setting, see Section 3.

With these two reductions, we give our main efforts to the weak type $(1,1)$ maximal estimate of lacunary sequence of truncated Calder\'on-Zygmund operators. The main ingredient is a new noncommutative Calder\'on-Zygmund decomposition communicated to us by Cadilhac \cite{C2}, see Theorem \ref{czdecom}. This decomposition, compared to Parcet's one \cite{JP1}, admits a great advantage that
the off-diagonal part of the good function vanishes (see Remark \ref{parcet});
recall that in his long paper \cite{JP1}, in order to deal with this part, Parcet had to exploit a pseudo-localization principle which constitutes a major part of that paper (see \cite{C} for a simplified proof of this principle). In our present case for noncommutative maximal estimate, it is not clear at all that whether there exists a pseudo-localization principle. To make this new decomposition work, we are also partially inspired by Cadilhac's note \cite{C2} where he deals with Lipchitz kernels, see Remark \ref{clever}.


The rest of the paper is organized as follows. In Section 2, we review some facts on the $\ell_\infty$-valued noncommutative $L_{p}$ spaces as well as the noncommutative Calder\'on-Zygmund decomposition found by Cadilhac. Section 3 and 4 are devoted to the proof of Theorem \ref{p2}. In Section 5 (resp. Section 6), we give the proof of the maximal inequalities stated in Theorem \ref{t1} (resp. Theorem \ref{t3}). In the last section, we show the noncommutative pointwise convergence results stated in
Theorem \ref{t1} and Theorem \ref{t3}.

\textbf{Notation:} Throughout the paper we write
$X\lesssim Y$ for nonnegative quantities $X$ and $Y$ to mean that there is  some inessential constant $C>0$ such that $X\le CY$ and we write $X\thickapprox Y$ to imply that $X\lesssim Y$ and $Y \lesssim X$.

\medskip

NOTE: After the preliminary version of the paper was completed, the first author was kindly told by Javier Parcet that Theorem \ref{t2} has been discovered independently by him and Jos\'{e} M Conde-Alonso.
\section{Preliminaries}
\subsection{Noncommutative $L_{p}$-spaces}
Let $\M$ be a semifinite von Neumann algebra equipped with a \emph{n.s.f.} trace $\tau$.
Denote by $\M_{+}$ the positive part of $\M$ and let $\mathcal{S_{\M+}}$  be the set of all $x\in\M_{+}$ whose support projection has a finite trace. Let $\mathcal{S}_{\M}$ be the linear span of $\mathcal{S_{\M+}}$, then $\mathcal{S}_{\M}$ is a $w^{*}$-dense $\ast$-subalgebra of $\M$. Let $0< p<\infty$. For any $x\in\mathcal{S}_{\M}$, $|x|^{p}\in\mathcal{S}_{\M}$ and we set
$$\|x\|_{p}=\big(\tau(|x|^p)\big)^{1/p},\ \ x\in\mathcal{S}_{\M}.$$
Here $|x|=(x^{\ast}x)^{\frac{1}{2}}$ is the modulus of $x$. By definition, the noncommutative $L_{p}$-space associated with $(\M,\tau)$ is the completion of $(\mathcal{S}_{\M},\|\cdot \|_{p})$ and it is denoted by $L_{p}(\M)$. For convenience, we set $L_{\infty}(\M) = \M$ equipped with the operator norm $\|\cdot \|_{\M}$. Let $L_{p}(\M)_{+}$ denote the positive part of $L_{p}(\M)$.

Suppose that $\M\subset B(\mathcal{H})$ acts on a separable Hilbert space $\mathcal{H}$. Let $\M'$ be the commmutant of $\M$. A closed densely defined operator on $\mathcal{H}$ is said to be affiliated with $\M$ when it commutes with every unitary operator $u$ in $\M'$. If $x$ is a densely defined selfadjoint operator on $\mathcal{H}$
and $x = \int_{\R} \lambda \hskip1pt d \gamma_x(\lambda)$ is its corresponding spectral
decomposition, then the spectral projection $\int_{\mathcal{I}} d
\gamma_x(\lambda)$ will be simply denoted by $\chi_{\mathcal{I}}(x)$, where $\mathcal{I}$ is a measurable subset of $\R$. A closed and densely defined operator
$x$ affiliated with $\mathcal{M}$ is called \emph{$\tau$-measurable} if
there exists $\lambda > 0$ such that $$\tau \big( \chi_{(\lambda,\infty)}
(|x|) \big) < \infty.$$
We denote the set of the $\ast$-algebra of \emph{$\tau$-measurable} operators by $L_{0}(\M)$.
For $1\leq p<\infty$, the weak $L_{p}$-space $L_{p,\infty}(\M)$ is defined as the set of all $x$ in $L_0(\M)$ with finite quasi-norm
$$\|x\|_{p,\infty}=\sup_{\lambda > 0}\lambda\tau \big( \chi_{(\lambda,\infty)}
(|x|) \big)^{\frac{1}{p}}<\infty.$$
We refer the reader to \cite{P2,FK} for a detailed exposition of noncommutative $L_{p}$-spaces.

\medskip

\subsection{Vector-valued noncommutative $L_{p}$-spaces}
We first recall the column space. Let $(\Sigma,\mu)$ be a measure space. The column space $L_p(\M;L^c_2(\Sigma))$ consists of the operator-valued functions $f$ with finite norm for $p\geq1$ (quasi-norm for $0<p<1$)
$$\|f\|_{L_p(\M;L^c_2(\Sigma))}=\Big\|\Big(\int_{\Sigma}f^*(\omega)f(\omega)d\mu(\omega)\Big)^{\frac12}\Big\|_{p}<\infty.$$
We refer the reader to \cite{P2} for precise definition and related properties of the Hilbert valued operator spaces.
The most important property for our purpose is the following H\"older type inequality (see e.g. \cite[Proposition 1.1]{M}).
\begin{lem}\label{mainlemma}\rm
Let $0<p,q,r\leq\infty$ be such that $1/r=1/p+1/q$. Then for any $f\in L_p(\M;L^c_2(\Sigma))$ and $g\in L_q(\M;L^c_{2}(\Sigma))$
$$\Big\|\int_{\Sigma}f^*(\omega)g(\omega)d\mu(\omega)\Big\|_{r}\leq \Big\|\Big(\int_{\Sigma}|f(\omega)|^2d\mu(\omega)\Big)^{\frac12}\Big\|_{p}
\Big\|\Big(\int_{\Sigma}|g(\omega)|^2d\mu(\omega)\Big)^{\frac12}\|_{q}.$$
\end{lem}

\medskip

As we mentioned in the introduction, a fundamental objective of this paper is the $\ell_\infty$-valued noncommutative $L_{p}$ spaces $L_p(\mathcal {M};\ell_\infty)$ introduced by Pisier \cite{pisier} and Junge \cite{junge}. Given $1\le p\leq\infty$,
we define $L_p(\mathcal {M};\ell_\infty)$ as the space of all sequences $x=(x_k)_{k\geq1}$
in $L_p(\mathcal {M})$ which admits a factorization of the following form:
there exist $a,b\in L_{2p}( \mathcal {M})$  and a bounded sequence $y=(y_k)_{k\geq1}\subset L_\infty(\mathcal {M})$ such that
$$x_k=ay_kb,\ \ \ k\ge 1.$$
The norm of $x$ in $L_p(\mathcal {M};\ell_\infty)$ is defined as
$$
\|x\|_{L_p(\mathcal {M};\ell_\infty)}=\inf\left\{\big\|a\big\|_{{2p}}\sup_k
\big\|y_k\big\|_\infty\big\|b\big\|_{{2p}}\right\},$$
where the infimum is taken over all factorizations of $x$ as above. It is easy to verify that $L_p(\mathcal {M};\ell_\infty)$ is a Banach space equipped with the norm $\|\cdot \|_{L_p(\mathcal {M};\ell_\infty)}$. As usual, the norm of $x$ in $L_p(\mathcal {M};\ell_\infty)$ is conventionally denoted by $\|{\sup_{k\geq1}}^+x_k\|_p$. However, we should point out that
${\sup_{k\geq1}}^+x_k$ is
just a notation since ${\sup_{k\geq1}}x_k$ does not make any sense
in the noncommutative setting. We just use this notation for convenience.
More generally, for any index set $I$, the space $L_p(\mathcal {M}; \ell_\infty(I))$ can be defined similarly.
The following property can be found in \cite[Remark 4.1]{CXY13}.
\begin{rk}\label{rk:MaxFunct}\rm
Let $x=(x_k)_{k\in I}$ be a sequence of selfadjoint operators in $L_p(\M)$.
Then $x\in L_p(\M;\ell_\infty)$ iff there is a positive operator $a\in
L_p(\M)$ such that $-a\leq x_k \le a$ for all $k\in I$, and moreover,
$$\big \|{\sup_{k\in I}}^{+} x_k \big \|_p=\inf\big\{\|a\|_p\;:\; a\in L_p(\M)_{+},\; -a\leq x_k\le a,\;\forall\; k\in I\big\}.$$
\end{rk}

Besides the strong maximal norm which corresponds to the $L_{p}$-norm of a maximal function, we now turn to the weak maximal norm (see e.g. \cite{Hong20}) which corresponds to the weak
$L_{p}$ norm of a maximal function. Let $I$ be an index set. Given a family $(x_k)_{k\in I}$
in $L_p(\mathcal {M})$ with $1\leq p< \infty$, we define
\begin{align}\label{weakm}
\|(x_k)_{k\in I}\|_{\Lambda_{p,\infty}(\M;\ell_{\infty}(I))}
=\sup_{\lambda>0}\lambda\inf_{e\in\mathcal{P}(\M)}
\Big\{\big(\tau(e^{\perp})\big)^{\frac{1}{p}}:\|ex_ke\|_{\infty}\leq\lambda\ \mbox{for\ all}\ k\in I\Big\},
\end{align}where $\mathcal{P}(\M)$ is the set of all projections in $\M$. Finally, we set the quasi-Banach space $\Lambda_{p,\infty}(\M;\ell_{\infty}(I))$ to be the set of
all sequences $x=(x_k)_{k\in I}$ in $L_{p,\infty}(\M)$ such that its $\Lambda_{p,\infty}(\M;\ell_{\infty}(I))$ quasi-norm is finite.
We will omit the index set $I$ when it will not cause confusions.

\begin{rk}\label{weakm4}\rm
Let $x=(x_k)_{k}$ be a sequence of selfadjoint operators in $L_{p}(\M)$. As in Remark \ref{rk:MaxFunct}, there is a similar characterization of $\Lambda_{p,\infty}(\M;\ell_{\infty})$ quasi-norm for a sequence of selfadjoint operators $x=(x_k)_{k}$,
\begin{align}\label{weakm1}
\|(x_k)_{k}\|_{\Lambda_{p,\infty}(\M;\ell_{\infty})}=\sup_{\lambda>0}\lambda\inf_{e\in\mathcal{P}(\M)}
\Big\{\big(\tau(e^{\perp})\big)^{\frac{1}{p}}:-\lambda\leq ex_ke\leq\lambda\ \mbox{for\ all}\ k\Big\}.
\end{align}
Indeed, this follows from
\begin{align}\label{weakm7}
-\lambda\leq ex_ke\leq\lambda \Leftrightarrow |ex_ke|\leq\lambda \Leftrightarrow \|ex_ke\|_{\infty}\leq\lambda
\end{align}
for any $k$.
\end{rk}
\begin{remark}\rm
The equivalence relationships in (\ref{weakm7}) would not be true in general if $\lambda$ is replaced by a positive operator. More precisely, considering a positive operator $g$ and a selfadjoint operator $f$ such that $-g\leq f\leq g$, it is not true that $|f|\leq g$ in general. For instance, consider
$$f=\left( \begin{array}{ll}
   8 & \ 0 \\
   0 &-8
 \end{array}\right) ,\quad
 g=\left( \begin{array}{ll}
   10 & \ 6 \\
    \ 6 &10
   \end{array}\right);
 $$
then it is easy to see that  $-g\leq f\leq g$, while $g-|f|$ is not positive.
\end{remark}

\subsection{Noncommutative Calder{\'o}n-Zygmund decomposition}\quad


\medskip

The noncommutative Calder{\'o}n-Zygmund decomposition is based on Cuculescu's construction for the standard dyadic martingales. Let us recall briefly the related notions.
Given an integer $k \in \Z$, $\Q_k$ stands for the set of dyadic cubes of side length $2^{-k}$, $\sigma_k$ be the $\sigma$-algebra generated by $\Q_k$ and $\mathcal N_k=L_\infty(\mathbb R^d,\sigma_k,dx)\overline{\otimes}\M$ be the associated von Neumann subalgebra of $\mathcal N$, where $\mathcal N=L_\infty(\mathbb R^d)\overline{\otimes}\M$ was given in the introduction. Then it is well-known that $(\mathcal{N}_k)_{k \in \Z}$ is a sequence of increasing von Neumann subalgebras such that the union is weak$^*$ dense in $\mathcal N$, and thus forms a filtration with the associated conditional expectations $(\mathsf{E}_k)_{k\in\Z}$ defined as
$$\mathsf{E}_k(f) = \sum_{Q \in \Q_k}^{\null} f_Q \chi_Q,\;\forall f\in L_1(\mathcal N)$$
where $\chi_Q$ is the characteristic function of $Q$ and $f_Q$ denotes the mean of $f$ over $Q$
$$f_Q = \frac{1}{|Q|} \int_Q f(y) \, dy.$$
Here $|Q|$ denotes the volume of $Q$.

Let $1\leq p\leq\infty$. A sequence $(f_k)_{k\in\mathbb Z}\subset L_p(\mathcal N)$ will be called a $L_p$-martingale if it satisfies $\mathsf{E}_{k-1}(f_k)=f_{k-1}$; in this case, the martingale difference is defined as $df_{k}=f_{k}-f_{k-1}$.

For further convenience, we make some conventions or definitions. Denote by  $\Q$ the set of all standard dyadic cubes in $\R^d$. The notation $dist(x,Q)$ means the distance between $x$ and $Q$.
For all $x\in\R^{d}$, denote by $Q_{x,k}$ the unique cube in $\Q_k$ containing $x$ and $c_{x,k}$ denotes its centre. Let $Q\in\Q$ and $i$ be any odd positive integer, $iQ$ stand for the cube with the same center as $Q$ such that $\ell(iQ)=i\ell(Q)$, where $\ell(Q)$ denotes the side length of $Q$.

\medskip

As in \cite{JP1}, the noncommutative Calder\'on-Zygmund decomposition will be constructed for functions in the class
$$\mathcal N_{c,+} =\Big\{
f: \R^d \to \M\cap L_1(\M) \, \big| \ f \geq0, \
\overrightarrow{\mathrm{supp}} \hskip1pt f \ \ \mathrm{is \
compact} \Big\},$$
which is dense in $L_1(\mathcal N)_{+}$. Recall that
$\overrightarrow{\mathrm{supp}}$ means the support of $f$ as an
operator-valued function in $\R^d$, which is different from its support projection as an element
of a von Neumann algebra.
\begin{Cuculescutheo}\rm
Let $f \in \mathcal{N}_{c,+}$ and $\lambda>0$. Denote $f_k=\mathsf{E}_k(f)$. It was shown in \cite[Lemma 3.1]{JP1} that there exists $m_{\lambda}(f)\in\Z$ such that $f_{k}\leq\lambda\1_{\mathcal{N}}$ for all $k\leq m_{\lambda}(f)$. Now by adopting Cuculescu's construction \cite{Cuc} to $(f_k)_{k\in\mathbb Z}$ relative to the dyadic filtration $(\mathcal{N}_k)_{k \in\Z}$, one can find a
sequence of decreasing projections $(q_k)_{k\in\Z}$ defined recursively by $q_k = \1_\mathcal{N}$ for $k\leq m_{\lambda}(f)$ and for $k>m_\lambda(f)$
$$q_k=q_k(f,\lambda)=\chi_{(0,\lambda]}(q_{k-1} f_k q_{k-1})$$
such that
\begin{itemize}
\item $q_k$ commutes with $q_{k-1} f_k
q_{k-1}$;

\item $q_k$ belongs to $\mathcal{N}_k$ and $q_k f_k q_k \le \lambda \hskip1pt
q_k$;

\item the following estimate holds $$\varphi \Big(
\mathbf{1}_\mathcal{N} - \bigwedge_{k \in\Z} q_k \Big) \le
\frac{\|f\|_1}{\lambda}.$$
\end{itemize}
In the present semi-commutative setting, $q_k$ admits the following expression
$$q_{k}=\sum_{Q\in\Q_{k}}q_{Q}\chi_{Q},$$
where $q_{Q}$ is a projection in $\M$ with
$$q_{Q}=\begin{cases} \1_\M & \mbox{if} \ k \leq
m_{\lambda}(f),
\\ \chi_{(0,\lambda]} \big( q_{\widehat{Q}} f_Q q_{\widehat{Q}}\big) & \mbox{if} \ k > m_{\lambda}(f),
\end{cases}$$
where $\widehat{Q}$ is the dyadic father of $Q$. Accordingly these projections satisfy
\begin{equation}\label{czd5}
\  q_Q\leq q_{\widehat{Q}},\ \
\  q_Q\  \mbox{commutes\ with}\  q_{\widehat{Q}} f_Q q_{\widehat{Q}},\ \
\   q_Q f_Q q_Q \le \lambda q_Q.
\end{equation}

Then one can define the sequence $(p_k)_{k \in \Z}$ of disjoint projections by $p_k =
q_{k-1}-q_k$ such that $$\sum_{k \in \Z} p_k = \1_\mathcal{N} - q =q^\perp\quad \mbox{with} \quad q = \bigwedge_{k \in
\Z} q_k.$$
One can then express the projections $p_k$ as
\begin{equation}\label{czd1}
p_k=\sum_{Q\in\Q_{k}}(q_{\widehat{Q}}-q_{Q})\chi_{Q}\triangleq\sum_{Q\in\Q_{k}}p_Q\chi_{Q}.
\end{equation}
\end{Cuculescutheo}
The following version of noncommutative Calder{\'o}n-Zygmund decomposition was communicated to us by Cadilhac \cite{C2}. 
\begin{thm}\label{czdecom}\rm
Fix $f\in\mathcal N_{c,+}$, $\lambda>0$ and $s\in\N$. Let $(q_k)_{k\in\Z}$ and $(p_k)_{k\in\Z}$ be the two sequences of projections appeared in the above Cuculescu's construction. Then there exist a projection $\zeta\in \mathcal N$ defined by
\begin{equation}\label{czd9}
\zeta = \big(\bigvee_{Q\in \Q} p_Q\chi_{(2s+1)Q}\big)^{\bot},
\end{equation}
and a decomposition of $f$,
\begin{equation}\label{czd76789}
f = g + b_d +
b_\mathit{off}
\end{equation}
such that the following assertions hold.

\begin{enumerate}[(1)]
\item [(i)]$\varphi(\mathbf 1_\mathcal{N}-\zeta) \leq (2s+1)^{d}\dfrac{\|f\|_1}{\lambda}$, where $\1_{\mathcal{N}}$ stands for the unit
elements in $\mathcal{N}$.
\item[(ii)] $g$ is positive in $\mathcal{N}$; moreover,
$\| g\|_1 \le\|f\|_1 \quad \mbox{and} \quad \|
g\|_\infty \le 2^{d} \lambda$.


\item[(iii)]$b_d=\sum_{n\in\Z}b_{d,n}$, where
\begin{equation}\label{czd6}
b_{d,n}=p_n(f-f_n) p_n.
\end{equation}
Each $b_{d,n}$ satisfies the cancellation conditions: for $Q\in \Q_{n}$,
     $\int_Q b_{d,n} = 0;$
    and for all $x,y\in\R^{d}$ such that $y \in (2s+1)Q_{x,n}$, $\zeta(x)b_{d,n}(y)\zeta(x) = 0$. Furthermore, $\sum_{n\in\Z}\|
b_{d,n} \|_1 \le 2 \, \|f\|_1$.


\item[(iv)]$b_\mathit{off}=\sum_{n\in\Z}b_{n}$, where
\begin{equation}\label{czd7}
b_{n}=p_n (f-f_{n}) q_n+q_n(f-f_{n})p_n.
\end{equation}
Each $b_{n}$ satisfies: for $Q\in \Q_{n}$,
     $\int_Q b_{n} = 0$; and for all $x,y\in\R^{d}$ such that $y \in (2s+1)Q_{x,n}$, $\zeta(x)b_{n}(y)\zeta(x) = 0$.
\end{enumerate}
\end{thm}

\begin{remark}\label{parcet}\rm
The details of the proof of Theorem \ref{czdecom} will appear in a forthcoming joint paper by Cadilhac, Conde-Alonso and Parcet, and Theorem \ref{czdecom} is one of the main results there. The new input in Cadilhac's decomposition (\ref{czd76789}) is to replace the maximum $i \vee j=\max(i,j)$ in
Parcet's decomposition (see \cite[(3.3)]{JP1}) by the minimum $i\wedge j=\min(i,j)$; one can then check easily that $g_\mathit{off}$ vanishes.
\end{remark}
\begin{remark}\label{comp1}\rm
For the purpose in this paper (see, e.g. (\ref{bad34789})), we will fix $s=4[\sqrt{d}]$ in the rest of the paper, where $[l]$ denotes the integer part of $l$.
\end{remark}

\section{Proof of Theorem \ref{p2}: two reductions}
To prove Theorem \ref{p2}, we give the two reductions in the present section.
\subsection{Reduction to real Calder\'on-Zygmund kernel}
\quad
\vskip0.24cm

Given a CZO $T$ and its truncated ones $T_\varepsilon$ with kernel $k$, the real and imaginary parts of the kernel are denoted by $\mathrm{Re}(k)$ and $\mathrm{Im}(k)$ respectively; accordingly, the corresponding operators are denoted by
 $$\mathrm{Re}(T_{\varepsilon})f(x)=\int_{|x-y|>\varepsilon}\mathrm{Re}(k)(x,y)f(y)dy$$\ \ and$$ \mathrm{Im}(T_{\varepsilon})f(x)=\int_{|x-y|>\varepsilon}\mathrm{Im}(k)(x,y)f(y)dy.$$

Due to the following observation,  we are able to assume that all the kernels are real, which will play an essential role in establishing the noncommutative maximal estimates.

\begin{lem}\label{mainlemma2}\rm
Let $T$ be a CZO with associated kernel $k$ satisfying (\ref{1}) and (\ref{6}) with $q=2$ and let $T_{\varepsilon}$ be defined as (\ref{maximal}). Then we have following basic facts.
\begin{itemize}
\item[(i)] Both $\mathrm{Re}(k)$ and $\mathrm{Im}(k)$ satisfy (\ref{1}) and (\ref{6}) with $q=2$.

\item[(ii)]  If $(T_{\varepsilon})_{\varepsilon>0}$ is of strong type $(p_0,p_0)$ for some $p_{0}\in(1,\infty)$, then both $(\mathrm{Re}(T_{\varepsilon}))_{\varepsilon>0}$ and $(\mathrm{Im}(T_{\varepsilon}))_{\varepsilon>0}$ are of strong type $(p_{0},p_{0})$.

\item[(iii)]  If $(\mathrm{Re}(T_{\varepsilon})f)_{\varepsilon>0}$ and $(\mathrm{Im}(T_{\varepsilon})f)_{\varepsilon>0}$ are of weak type $(1,1)$, so is $(T_{\varepsilon}f)_{\varepsilon>0}$.
\end{itemize}
\end{lem}
\begin{proof}
(i) Note that $\max\{|\mathrm{Re}(k)|,|\mathrm{Im}(k)|\}\leq|k|$, it is easy to verify that $\mathrm{Re}(k)$ and $\mathrm{Im}(k)$ satisfy the conditions (\ref{1}) and (\ref{6}) with $q=2$.

(ii) We only consider $(\mathrm{Re}(T_{\varepsilon}))_{\varepsilon>0}$ since the argument for $(\mathrm{Im}(T_{\varepsilon}))_{\varepsilon>0}$ is similar. Let $f\in L_{p_0}(\mathcal N)$.
To estimate $L_{p_0}(\mathcal N;\ell_\infty)$-norm of $(\mathrm{Re}(T_{\varepsilon})f)_{\varepsilon>0}$, by the triangle inequalities and the facts that $L_{p_0}$-norm of $\mathrm{Re}(f)$ and $\mathrm{Im}(f)$ are dominated by $\|f\|_{p_0}$, it suffices to assume $f=f^*$. Then the claim follows from $\mathrm{Re}(T_{\varepsilon})f=\frac{1}{2}\big(T_{\varepsilon}f+(T_{\varepsilon}f)^{\ast}\big)$  for any $\varepsilon>0$ and the fact that the adjoint map is an isometric isomorphism on $L_{p_0}(\mathcal N;\ell_\infty)$.

(iii) This assertion follows from $T_{\varepsilon}f=\mathrm{Re}(T_{\varepsilon})f+i\mathrm{Im}(T_{\varepsilon})f$ for any $\varepsilon$ and the quasi-norm of $\Lambda_{1,\infty}(\mathcal N; \ell_\infty)$.
\end{proof}


\subsection{Reduction to the maximal estimates of lacunary subsequence}
\quad
\vskip0.24cm

In this subsection, we reduce the study of $(T_{\varepsilon})_{\varepsilon>0}$ to its lacunary subsequence.

Let
$\phi$ be  a smooth radial nonnegative function  on $\R^{d}$ such that $\mathrm{supp}\;\phi\subset\{x\in \R^{d}:1/2\leq|x|\leq2\}$ and $\sum_{i\in\Z}\phi_{i}(x)=1$ for all $x\in\R^{d}\setminus \{0\}$, where $\phi_{i}(x)=\phi(2^{i}x/\sqrt{d})$.
This $\phi$ will be fixed in the whole paper.
Consequently, for a reasonable $f$,
$$T_{\varepsilon}f(x)=\sum_{i\in\Z}\int_{|x-y|>\varepsilon}k(x,y)\phi_{i}(x-y)f(y)dy.$$
For the sake of convenience, let $k^{\phi}_{i}(x,y)$ denote the kernel $k(x,y)\phi_{i}(x-y)$ and set $\Delta_{i}=\{x\in \R^{d}:2^{-i-1}\sqrt{d}\leq|x|\leq2^{-i+1}\sqrt{d}\}$. Then $\mathrm{supp}\;\phi_{i}\subset\Delta_{i}$. Note that
$$\int_{|x-y|>\varepsilon}k^{\phi}_{i}(x,y)f(y)dy$$
vanishes if the intersection of $x+\Delta_{i}$ and $\{y\in \R^{d}:|x-y|>\varepsilon\}$ is empty. This implies $2^{-i+1}\sqrt{d}>\varepsilon$ and thus $i<\log_{2}(\frac{2\sqrt{d}}{\varepsilon})$ by a simple calculation. Set $j_\varepsilon\triangleq[\log_{2}(\frac{2\sqrt d}{\varepsilon})]$. With these conventions and observations, we may write $T_{\varepsilon}f(x)$ as
\begin{equation}\label{max191}
T_{\varepsilon}f(x)=\sum_{i:i\leq j_\varepsilon-1}
\int_{\R^{d}}k^{\phi}_{i}(x,y)f(y)dy
+\int_{|x-y|>\varepsilon}k^{\phi}_{j_\varepsilon}(x,y)f(y)dy.
\end{equation}
For convenience, let us write the second term above as $T^{\phi}_{\varepsilon,j_\varepsilon}f(x)$.
\begin{prop}\label{HLX1}\rm
Let $T$ be a CZO with kernel $k$ satisfying the assumptions in Theorem \ref{p2}. Then
\begin{itemize}
\item[(i)] $(T^{\phi}_{\varepsilon,j_\varepsilon})_{\varepsilon>0}$ is of weak type $(1,1)$. More precisely,
for any $\lambda>0$ and $f\in L_{1}(\mathcal{N})$, there exists a projection $\eta\in \mathcal N$ such that
$$\sup_{\varepsilon>0}\big\|\eta (T^{\phi}_{\varepsilon,j_\varepsilon}f)\eta\big\|_\infty\lesssim\lambda\quad\mbox{and}\quad \varphi (\eta^{\perp})\lesssim \frac{\|f\|_1}{\lambda};$$

\item[(ii)] $(T^{\phi}_{\varepsilon,j_\varepsilon})_{\varepsilon>0}$ is of strong type $(p,p)$ for all $1<p\leq\infty$.
\end{itemize}
\end{prop}
\begin{proof}
(i) By the quasi-triangle inequality, we can assume that $f$ is positive; by Lemma \ref{mainlemma2}, $k$ can be assumed to be real.

Applying the size condition (\ref{1}) of the kernel $k$ as well as the support of $\phi_{j_{\varepsilon}}$, we find that for any $\varepsilon>0$
\begin{align*}
  T^{\phi}_{\varepsilon,j_\varepsilon}f(x)& \leq\int_{|x-y|>\varepsilon}|k^{\phi}_{j_\varepsilon}(x,y)|f(y)dy\leq\int_{|x-y|\leq2^{-j_\varepsilon+1}\sqrt d}|k^{\phi}_{j_\varepsilon}(x,y)|f(y)dy\\
  & \lesssim\int_{|x-y|\leq2^{-j_\varepsilon+1}\sqrt d}\frac{\phi_{j_\varepsilon}(x-y)}{|x-y|^{d}}f(y)dy
  \lesssim M_{2^{-j_\varepsilon+1}\sqrt d} f(x),
\end{align*}
where $M_r$ is the Hardy-Littlewood averaging operator.  On the other hand, note that
$$T^{\phi}_{\varepsilon,j_\varepsilon}f(x) \geq-\int_{|x-y|>\varepsilon}|k^{\phi}_{j_\varepsilon}(x,y)|f(y)dy,$$
thus one gets
$-M_{2^{-j_\varepsilon+1}\sqrt d}f(x)\lesssim T^{\phi}_{\varepsilon,j_\varepsilon}f(x).$ Hence,
\begin{align}\label{exaq1}
-M_{2^{-j_\varepsilon+1}\sqrt d}f(x)\lesssim T^{\phi}_{\varepsilon,j_\varepsilon}f(x)\lesssim M_{2^{-j_\varepsilon+1}\sqrt d}f(x).
\end{align}
We now appeal to Mei's noncommutative Hardy-Littlewood maximal weak type $(1,1)$ inequality \cite{M}: there exists a projection $\eta\in \mathcal{N}$ such that for any $\varepsilon>0$
$$\varphi(\eta^{\perp})\lesssim \frac{\|f\|_{1}}{\lambda}\ \ \mbox{and}\ \ \eta M_{2^{-j_{\varepsilon}+1}\sqrt d}f\eta\leq\lambda.$$
We then deduce that for any $\varepsilon>0$,
$$-\lambda\leq-\eta M_{2^{-j_{\varepsilon}+1}\sqrt d}f\eta\lesssim \eta T^{\phi}_{\varepsilon,j_\varepsilon}f\eta\lesssim \eta M_{2^{-j_{\varepsilon}+1}\sqrt d}f\eta\leq\lambda.$$
This, by Remark \ref{weakm4}, gives the desired weak type $(1,1)$ maximal estimate of $(T^{\phi}_{\varepsilon,j_\varepsilon})_{\varepsilon>0}$.

(ii) By noting Remark \ref{rk:MaxFunct}, the strong type $(p,p)$ estimate of $(T^{\phi}_{\varepsilon,j_\varepsilon})_{\varepsilon>0}$ is a consequence of (\ref{exaq1}) and Mei's noncommutative Hardy-Littlewood maximal strong type $(p,p)$ ($1<p\leq\infty$) inequality \cite{M}.
\end{proof}

By the quasi-triangle inequality in $\Lambda_{1,\infty}(\mathcal{N},\ell_{\infty})$ and Proposition \ref{HLX1}, to establish Theorem \ref{p2}, it suffices to prove the desired maximal weak type $(1,1)$ result of the first term on the right-hand side of (\ref{max191}). For convenience, denote the operator $T^\phi_{j}$ by
\begin{equation}\label{0090}
T^\phi_{j}f(x)=\sum_{i:i< j}\int_{\R^{d}}k^{\phi}_{i}(x,y)f(y)dy.
\end{equation}
\begin{thm}\label{HLX2}\rm
Let $T$ be a CZO with kernel $k$ satisfying the assumptions in Theorem \ref{p2} and let $T^\phi_{j}$ be defined as (\ref{0090}).  Then the sequence of linear operators $(T^\phi_{j})_{j\in\mathbb Z}$ is of weak type $(1,1)$. More precisely, for any $f\in L_1(\mathcal N)$ and $\lambda>0$, there exists a projection $e\in\mathcal{N}$ such that
\begin{align*}
\sup_{j\in\Z}\|e(T^{\phi}_{j}f) e\|_\infty\lesssim\lambda\ \ \text{and}
\ \ \varphi(e^{\perp})\lesssim\frac{\|f\|_1}{\lambda}.
\end{align*}
\end{thm}

\section{Proof of Theorem \ref{p2}: maximal estimate of lacunary subsequence $(T^\phi_j)_j$}

In this section, we will complete the proof of Theorem \ref{p2} by showing Theorem \ref{HLX2}.

According to Lemma \ref{mainlemma2}, we suppose that the kernel $k$ is real. By decomposing any element into linear combination of four  positive elements and recalling that $\mathcal N_{c,+}$ is dense in $L_1(\mathcal N)_{+}$ and by  the standard density argument, it suffices to show the maximal weak type $(1,1)$ estimate for $f\in \mathcal N_{c,+}$. Now fix one $f\in \mathcal N_{c,+}$ and a $\lambda\in(0,+\infty)$. Without loss of generality, we may assume $m_{\lambda}(f)=0$. By the noncommutative Calder\'on-Zygmund decomposition in Theorem \ref{czdecom}, we can decompose $f$ as $f=g+b_d+b_\mathit{off}$. Thus it suffices to find a projection $e\in \mathcal N$ such that
\begin{align*}
\forall j\in\mathbb Z,\;\max\Big\{\|e(T^{\phi}_{j}g) e\|_\infty,\;\|e(T^{\phi}_{j}b_d) e\|_\infty,\;\|e(T^{\phi}_{j}b_\mathit{off}) e\|_\infty\Big\}\leq\lambda\ \ \text{and}
\ \ \varphi(e^{\perp})\lesssim\frac{\|f\|_1}{\lambda},
\end{align*}
which will follow from, by setting $e=e_1\wedge e_2\wedge e_3$, the existences of three projections $e_1,e_2$ and $e_3$ such that
\begin{align}\label{HLX15}
\sup_{j\in\mathbb Z}\|e_{1} T^{\phi}_{j}g e_{1}\|_{\infty}\leq\lambda\ \ \mbox{and}\ \ \varphi(e_{1}^{\perp})\lesssim\frac{\| f\|_{1}}{\lambda},
\end{align}

\begin{align}\label{HLX7}
\sup_{j\in\mathbb Z}\|e_{2}T^{\phi}_{j}b_{d}e_{2}\|_{\infty}\leq\lambda\ \ \mbox{and}\ \ \varphi(e_{2}^{\perp})\lesssim\frac{\| f\|_{1}}{\lambda},
\end{align}
and
\begin{align}\label{HLX9}
\sup_{j\in\mathbb Z}\|e_{3}T^{\phi}_{j}b_\mathit{off}e_{3}\|_{\infty}\leq\lambda\ \ \mbox{and}\ \ \varphi(e_{3}^{\perp})\lesssim\frac{\| f\|_{1}}{\lambda}.
\end{align}


\subsection{Estimate for the good function $g$: \eqref{HLX15}}\quad

\medskip

The following lemma will be used in estimating $(T^{\phi}_{j}g)_{j\in\mathbb Z}$.
\begin{lem}\label{HLX31}\rm
The sequence of operators $(T^\phi_{j})_{j\in\mathbb Z}$ is of strong type $(p_0,p_0)$.
\end{lem}
\begin{proof}
From (\ref{max191}), for any $j\in\mathbb Z$, there exists one $\varepsilon=\varepsilon_j$ such that $j=j_\varepsilon$ and
$$T^{\phi}_{j} =
T_{\varepsilon_j}-T^{\phi}_{\varepsilon_{j},j_\varepsilon}.
$$
Note that the strong type $(p_0,p_0)$ of $(T_{\varepsilon_j})_{j\in\mathbb Z}$ (resp. $(T^{\phi}_{\varepsilon_{j},j_\varepsilon})_{j\in\mathbb Z}$) follows from the corresponding assumption (resp. conclusion (ii)) in Theorem \ref{p2} (resp. Proposition \ref{HLX1}). Thus by the triangle inequality, we finish the proof.
\end{proof}

Now we are ready to prove estimate \eqref{HLX15}.
\begin{proof}[Proof of estimate \eqref{HLX15}.]
The desired maximal weak type $(1,1)$ estimate for the diagonal part of the good function can be deduced from conclusion (ii) in Theorem \ref{czdecom} and Lemma \ref{HLX31}. Indeed, since $g$ is positive stated in conclusion (ii) in Theorem \ref{czdecom} and $k$ is real-valued, $T^{\phi}_{j}g$ is selfadjoint. Applying Lemma \ref{HLX31} and Remark \ref{rk:MaxFunct}, we can find a positive operator $a\in L_{p_{0}}(\mathcal{N})$ such that for any $j\in\Z$,
$$-a\leq T^{\phi}_{j}g\leq a\ \ \ \mbox{and}\ \ \ \|a\|_{p_0}\lesssim\|g\|_{p_0}.$$
Set $e_{1}=\chi_{(0,\lambda]}(a)$. Then for any $j\in\Z$,
$$-\lambda\leq-e_{1}ae_{1}\leq e_{1}T^{\phi}_{j}ge_{1}\leq e_{1}ae_{1}\leq\lambda.$$
On the other hand, by the Chebyshev and H\"{o}lder inequalities,
$$\varphi(e_{1}^{\perp})=\varphi(\chi_{(\lambda,\infty)}(a))\leq
\frac{\|a\|^{p_{0}}_{p_{0}}}{\lambda^{p_{0}}}\lesssim\frac{\|g\|^{p_{0}}_{p_{0}}}{\lambda^{p_{0}}}\leq \frac{\|g\|_1\|g\|^{p_{0}-1}_{\infty}}{\lambda^{p_{0}}}\lesssim\frac{\|f\|_{1}}{\lambda},$$
where in the last inequality we used conclusion (ii) stated in Theorem \ref{czdecom}. This, by Remark \ref{weakm4}, provides the desired estimate for $(T^{\phi}_{j}g)_{j\in\mathbb Z}$.
\end{proof}


\subsection{Estimate for the diagonal part of the bad function $b_d$: \eqref{HLX7}}\quad

\medskip

Using the projection $\zeta$ constructed in (\ref{czd9}), we decompose $T^{\phi}_{j}b_d$ in the following way $$T^{\phi}_{j}b_d= \zeta^\perp T^{\phi}_{j}b_d \zeta^\perp + \zeta \hskip1pt T^{\phi}_{j}b_d\zeta^\perp
+ \zeta^\perp T^{\phi}_{j}b_d\zeta + \zeta \hskip1pt T^{\phi}_{j}b_d
\zeta.$$
Then taking $e_{2,1}=\zeta$, we are reduced to finding a projection $e_{2,2}$ such that
\begin{align}\label{bd}
\sup_{j\in\mathbb Z}\|e_{2,2}\zeta T^{\phi}_{j}b_d\zeta e_{2,2}\|_{\infty}\leq\lambda\ \ \mbox{and}\ \ \varphi(e_{2,2}^{\perp})\lesssim\frac{\| f\|_{1}}{\lambda}.
\end{align}
Indeed, taking $e_2=e_{1,2}\wedge e_{2,2}$, then we obtain for any $j\in\Z$
$$\|e_{2}T^{\phi}_{j}b_de_{2}\|_\infty\leq \|e_{2,2}\zeta T^{\phi}_{j}b_d\zeta e_{2,2}\|_{\infty}\leq\lambda.$$
Together with (\ref{bd}), conclusion (i) in Theorem \ref{czdecom} and recalling that $s=4[\sqrt{d}]$ announecd in Remark \ref{comp1}, we get
$$\varphi(e_{2}^{\perp})\leq\varphi(e^{\perp}_{2,1})+\varphi(e^{\perp}_{2,2})\lesssim\frac{\| f\|_{1}}{\lambda}.$$
Thus we get estimate \eqref{HLX7}.





Define the operator $T_{\phi,i}$ as
\begin{align}\label{4}
T_{\phi,i}f(x)=\int_{\R^{d}}k^{\phi}_{i}(x,y)f(y)dy,
\end{align}
then $T^{\phi}_{j}=\sum_{i:i< j}T_{\phi,i}$. For $i,n\in\mathbb Z$ and $x,y\in\mathbb R^d$, we define
\begin{align}\label{4789}
k^{\phi}_{i,n}(x,y)=\big(k^{\phi}_{i}(x,y) - k^{\phi}_{i}(x,c_{y,n})\big).
\end{align}

\begin{lem}\label{HLX6}\rm
For $i<n-1$ and $1\leq q<\infty$, the following estimate holds
\begin{align}\label{bad34}
\sup_{y\in\R^{d}}\int_{\R^{d}}|k^{\phi}_{i,n}(x,y)|^{q}dx\lesssim2^{id(q-1)}\big(\delta^{q}_{q}(n-i)+
\delta^{q}_{q}(n-i+1)+2^{q(i-n)}\big),
\end{align}
where $\delta_q$ was defined in \eqref{5}.
\end{lem}

\begin{proof}
To see this, note that
\begin{align*}
\begin{split}
\int_{\R^{d}}|k^{\phi}_{i,n}(x,y)|^{q}dx&\lesssim
\int_{\R^{d}}|k(x,y)-k(x,c_{y,n})|^{q}|\phi_i(x-y)|^{q}dx\\
&\quad+\int_{\R^{d}}|k(x,c_{y,n})|^{q}|\phi_i(x-y)-\phi_i(x-c_{y,n})|^{q}dx.
\end{split}
\end{align*}
We first claim
\begin{align}\label{cz12345}
\begin{split}\int_{\R^{d}}|k(x,c_{y,n})|^{q}&|\phi_i(x-y)-\phi_i(x-c_{y,n})|^{q}dx\\
& \lesssim\int_{|x-c_{y,n}|\approx 2^{-i}}|k(x,c_{y,n})|^{q}2^{iq}|y-c_{y,n}|^{q}dx.
\end{split}
\end{align}
Indeed, fixing $y\in\mathbb R^d$, it suffices to consider, in the integral above of  \eqref{cz12345}, those $x$ such that $\phi_i(x-y)-\phi_i(x-c_{y,n})\neq0$. By the support of $\phi_{i}$, at least one of the following conditions hold
\begin{enumerate}
\item $2^{-i-1}\sqrt{d}\leq|x-y|\leq2^{-i+1}\sqrt{d}$;
\item $2^{-i-1}\sqrt{d}\leq|x-c_{y,n}|\leq2^{-i+1}\sqrt{d}.$
\end{enumerate}
If (2) holds, then we get $|x-c_{y,n}|\approx 2^{-i}$. If (1) holds, then by the definition of $c_{y,n}$, we see that $|y-c_{y,n}|\leq2^{-n-1}\sqrt{d}$. Moreover, since $i<n-1$,
$$|y-c_{y,n}|\leq2^{-i-2}\sqrt{d}<\frac{1}{2}|x-y|.$$
Thus we deduce that
$$|x-c_{y,n}|\leq|x-y|+|y-c_{y,n}|<\frac{3}{2}|x-y|<2^{-i}\cdot3\sqrt{d}.$$
On the other hand,
$$|x-c_{y,n}|\geq|x-y|-|y-c_{y,n}|>\frac{1}{2}|x-y|>2^{-i-2}\sqrt{d}.$$
Therefore, in this case we still have $|x-c_{y,n}|\approx 2^{-i}$, where the constant only depends on the dimension $d$.
Now we use the relation obtained above and the mean value theorem to get (\ref{cz12345}). This is precisely the claim.

Finally, by the smoothness condition (\ref{5}) and size condition (\ref{1}) of the kernel $k$, the support of $\phi_{i}$ and (\ref{cz12345}), we find
\begin{align*}
\begin{split}
\int_{\R^{d}}|k^{\phi}_{i,n}(x,y)|^{q}dx&\lesssim\int_{2^{-i+n}2^{-n-1}\sqrt d\leq|x-y|\leq2^{-i+n+1}2^{-n-1}\sqrt d}|k(x,y)-k(x,c_{y,n})|^{q}dx\\
&\quad+\int_{2^{-i+n+1}2^{-n-1}\sqrt d\leq|x-y|\leq2^{-i+n+2}2^{-n-1}\sqrt d}|k(x,y)-k(x,c_{y,n})|^{q}dx\\
&\quad+\int_{|x-c_{y,n}|\approx 2^{-i}}|k(x,c_{y,n})|^{q}2^{iq}|y-c_{y,n}|^{q}dx\\
&\lesssim2^{id(q-1)}\big(\delta^{q}_{q}(n-i)+
\delta^{q}_{q}(n-i+1)+2^{q(i-n)}\big),
\end{split}
\end{align*}
since $|y-c_{y,n}|\leq2^{-n-1}\sqrt{d}$. This gives the desired estimate.
\end{proof}

Now we are at a position to show estimate \eqref{bd}.
\begin{proof}[Proof of estimate \eqref{bd}.]
Since $m_{\lambda}(f)=0$, we have $p_n=0$ for all $n\leq 0$. By conclusion (iii) in Theorem \ref{czdecom},  $b_{d}=\sum_{n=1}^{\infty}b_{d,n}$ where $b_{d,n}=p_n(f-f_n)p_n$ as in (\ref{czd6}).

We claim that $\zeta T_{\phi,i}b_{d,n}\zeta=0$ unless $i<n-1$. Fix one $x\in\mathbb R^d$. Recalling that $s=4[\sqrt d]$ in Remark \ref{comp1} and by the second cancellation property of $b_{d,n}$-conclusion (iii) in Theorem \ref{czdecom}, one has
\begin{align}\label{bad34789}
\zeta(x)T_{\phi,i}b_{d,n}(x)\zeta(x)=\int_{\R^{d}}k^{\phi}_{i}(x,y)\chi_{y\notin (8[\sqrt d]+1)Q_{x,n}}\zeta(x)b_{d,n}(y)\zeta(x)dy.
\end{align}
Indeed, recalling the definition of $k^{\phi}_{i}$, it suffices to consider those $y$ in the integral above such that $\phi_i(x-y)\neq0$ and ${y\notin (8[\sqrt d]+1)Q_{x,n}}$.
Note that the support of $\phi_{i}$ implies $2^{-i-1}\sqrt{d}\leq|x-y|\leq2^{-i+1}\sqrt{d}$; while ${y\notin (8[\sqrt d]+1)Q_{x,n}}$ implies that $|x-y|>4\sqrt{d}\cdot2^{-n}$. Thus $\zeta(x)T_{\phi,i}b_{d,n}(x)\zeta(x)$ may be equal to 0 unless $2^{-i+1}\sqrt{d}>4\sqrt{d}\cdot2^{-n}$, that is, $i<n-1$.
This is precisely the claim.

Taking these observations into consideration, we deduce that for any $x\in\mathbb R^d$,
\begin{align*}
  \zeta(x)T^{\phi}_{j}b_{d}(x)\zeta(x)
  =\zeta(x)\sum_{n=1}^{\infty}\sum_{i:i< j;i<n-1}\int_{\R^{d}}k^{\phi}_{i}(x,y)b_{d,n}(y)dy\zeta(x).
\end{align*}
Furthermore, by applying the first cancellation property of $b_{d,n}$-conclusion (iii) stated in Theorem \ref{czdecom}, we get
\begin{eqnarray*}
\zeta(x)T^{\phi}_{j}b_{d}(x)\zeta(x) =\zeta(x) \sum_{n=1}^{\infty}\sum_{i:i< j;i<n-1}\Big( \int_{\R^{d}} k^{\phi}_{i,n}(x,y)
b_{d,n}(y) \, dy \Big) \zeta(x),
\end{eqnarray*}
where $k^{\phi}_{i,n}(x,y)$ was given in (\ref{4789}). Note that
$$\int_{\R^{d}} k^{\phi}_{i,n}(x,y)
b_{d,n}(y) \, dy$$
is a selfadjoint operator. Consequently,
\begin{align*}
\zeta(x)T^{\phi}_{j}b_{d}(x)\zeta(x) & \leq\zeta(x) \sum_{n=1}^{\infty}\sum_{i:i< j;i<n-1} \Big| \int_{\R^{d}} k^{\phi}_{i,n}(x,y)
b_{d,n}(y) \, dy \Big|\zeta(x)\\
       & \leq\zeta(x) \sum_{n=1}^{\infty}\sum_{i:i<n-1} \Big| \int_{\R^{d}} k^{\phi}_{i,n}(x,y)
b_{d,n}(y) \, dy \Big|\zeta(x)\\
       &\triangleq \zeta(x)F_{1}(x)\zeta(x).
\end{align*}
Thus for any $j\in\Z$
\begin{align}\label{F1}
-\zeta(x)F_{1}(x)\zeta(x)\leq \zeta(x)T^{\phi}_{j}b_{d}(x)\zeta(x)\leq \zeta(x)F_{1}(x)\zeta(x).
\end{align}
In the following, we show that
\begin{align}\label{F121}
\|F_{1}\|_{1}\lesssim\|f\|_{1}.
\end{align}
To see this, by using the Fubini theorem and the Minkowski inequality, we arrive at
\begin{align*}
  \|F_{1}\|_{1}
  & =\tau \int_{\R^{d}}\sum_{n=1}^{\infty}\sum_{i:i<n-1}\Big|\int_{\R^{d}} k^{\phi}_{i,n}(x,y)
b_{d,n}(y)dy \Big|dx\\
& = \sum_{n=1}^{\infty}\sum_{i:i<n-1}\int_{\R^{d}}\Big\|\int_{\R^{d}} k^{\phi}_{i,n}(x,y)
b_{d,n}(y)dy\Big\|_{L_{1}(\M)}dx\\
& \leq \sum_{n=1}^{\infty}\sum_{i:i<n-1}\int_{\R^{d}}\int_{\R^{d}} |k^{\phi}_{i,n}(x,y)|
\|b_{d,n}(y)\|_{L_{1}(\M)}dydx.
\end{align*}
Now exploiting (\ref{bad34}) in Lemma \ref{HLX6} and the Fubini theorem, we find
\begin{align*}
  \|F_{1}\|_{1}  & \leq\sum_{n=1}^{\infty}\sum_{i:i<n-1}\int_{\R^{d}}\|b_{d,n}(y)\|_{L_{1}(\M)}\int_{\R^{d}} |k^{\phi}_{i,n}(x,y)|
dxdy\\
  & \lesssim\sum_{n=1}^{\infty}
  \sum_{i:i<n-1}[\delta_{1}(n-i)+\delta_{1}(n-i+1)+2^{-n+i}]\int_{\R^{d}}\|b_{d,n}(y)\|_{L_{1}(\M)}dy\\
& \leq\sum_{n=1}^{\infty}\sum_{m=1}^{\infty}[2\delta_{1}(m)+2^{-m}]\int_{\R^{d}}\|b_{d,n}(y)\|_{L_{1}(\M)}dy\\
& \lesssim\sum_{n=1}^{\infty}\sum_{m=1}^{\infty}[2\delta_{2}(m)+2^{-m}]\|b_{d,n}\|_{L_{1}(\mathcal{N})},
\end{align*}
where in the last inequality we used the H\"{o}lder inequality.
Now we use (\ref{6}) with $q=2$ and conclusion (iii) announced in Theorem \ref{czdecom} to get
$$\|F_{1}\|_{1}\lesssim\sum_{n=1}^{\infty}\|b_{d,n}\|_{1}\lesssim\|f\|_{1},$$
which establishes (\ref{F121}).
As a consequence, set $e_{2,2}=\chi_{(0,\lambda]}(\zeta F_{1}\zeta)$. Then estimate \eqref{F1} implies for any $j\in\Z$
\begin{align*}
-\lambda\leq-e_{2,2}\zeta F_{1}\zeta e_{2,2}\leq e_{2,2}\zeta T^{\phi}_{j}b_{d}\zeta e_{2,2}\leq e_{2,2}\zeta F_{1}\zeta e_{2,2}\leq\lambda;
\end{align*}
moreover, by the Chebyshev inequality and (\ref{F121}), we find
$$ \varphi(e_{2,2}^{\perp})\leq\frac{\|\zeta F_{1}\zeta\|_{1}}{\lambda}\leq\frac{\| F_{1}\|_{1}}{\lambda}\lesssim\frac{\| f\|_{1}}{\lambda}.$$
This gives the desired estimate (\ref{bd}) thanks to Remark \ref{weakm4}.
\end{proof}

\subsection{Estimate for the off-diagonal part of the bad function $b_\mathit{off}$: \eqref{HLX9}}\quad

\medskip

Let us now turn to the off-diagonal term $b_\mathit{off}$.
Using the same argument as for $T^{\phi}_{j}b_{d}$ and taking $e_{3,1}=\zeta$, we are reduced to finding a  projection $e_{3,2}\in\mathcal{N}$ such that
\begin{align}\label{bd11}
\sup_{j\in\mathbb Z}\|e_{3,2}\zeta T^{\phi}_{j}b_\mathit{off}\zeta e_{3,2}\|_{\infty}\leq\lambda\ \ \mbox{and}\ \ \varphi(e_{3,2}^{\perp})\lesssim\frac{\| f\|_{1}}{\lambda}.
\end{align}
\begin{lem}\label{HLX13}\rm
For $i<n-1$, the following estimates hold for any $Q\in\mathcal Q_n$,
\begin{equation}\label{cz1009}
\begin{split}
\int_{\R^{d}}\Big\|\int_{Q}&|k^{\phi}_{i,n}(x,y)
|^{2}p_{Q} f(y) p_{Q} dy\Big\|^{\frac{1}{2}}_{L_{\frac{1}{2}}(\M)}dx\\
&\lesssim\Big((\delta^{2}_{2}(n-i)+\delta^{2}_{2}(n-i+1)+2^{2(i-n)})\tau(p_{Q})
\varphi(fp_{Q} \chi_{Q})\Big)^{\frac{1}{2}},
\end{split}
\end{equation}
\begin{equation}\label{cz10092}
\begin{split}
\int_{\R^{d}}\Big\|\int_{Q}&|k^{\phi}_{i,n}(x,y)
|^{2}p_{Q} f_Q p_{Q} dy\Big\|^{\frac{1}{2}}_{L_{\frac{1}{2}}(\M)}dx\\
&\lesssim\Big((\delta^{2}_{2}(n-i)+\delta^{2}_{2}(n-i+1)+2^{2(i-n)})\tau(p_{Q})
\varphi(f_Qp_{Q} \chi_{Q})\Big)^{\frac{1}{2}}.
\end{split}
\end{equation}
\end{lem}
\begin{proof}
Fix $i<n-1$ and $Q\in\mathcal Q_n$. We first consider \eqref{cz1009}. Let $x\in\R^{d}$. We claim that
$$\int_{Q}|k^{\phi}_{i,n}(x,y)
|^{2}p_{Q} f(y) p_{Q} dy=\int_{Q}|k^{\phi}_{i,n}(x,y)
|^{2}p_{Q} f(y) p_{Q} dy\chi_{dist(x,Q)\thickapprox2^{-i}}.$$
Indeed, the left integral may not be 0 only if there exists $y_0\in Q$ such that $k^\phi_{i,n}(x,y_0)\neq0$. By the definition of $k^\phi_{i,n}$, the support of $\phi_{i}$ implies that at least one of the following cases holds
\begin{enumerate}
\item $2^{-i-1}\sqrt{d}\leq|x-y_0|\leq2^{-i+1}\sqrt{d}$;
\item $2^{-i-1}\sqrt{d}\leq|x-c_{y_0,n}|\leq2^{-i+1}\sqrt{d}$.
\end{enumerate}
If (1) holds, then we use $i<n-1$ to get
$$dist(x,Q)\geq|x-y_0|-\sqrt{d}\ell(Q)\geq2^{-i-1}\sqrt{d}-2^{-i-2}\sqrt{d}=2^{-i-2}\sqrt{d}.$$
On the other hand, it is easy to see
$$dist(x,Q)\leq|x-y_0|\leq2^{-i+1}\sqrt{d}.$$
Thus we obtain $dist(x,Q)\thickapprox2^{-i}$. The same reason applies to case (2).
This is precisely the claim.

Now we apply the H\"{o}lder and Cauchy-Schwarz inequalities to get
\begin{align*}
&\ \int_{\R^{d}}\Big\|\int_{Q}|k^{\phi}_{i,n}(x,y)|^{2}p_{Q} f(y) p_{Q} dy\Big\|^{\frac{1}{2}}_{L_{\frac{1}{2}}(\M)}dx\\ &
\leq \int_{\R^{d}}\big\|p_{Q}\chi_{dist(x,Q)\thickapprox2^{-i}}\big\|^{\frac{1}{2}}_{L_{1}(\M)}
\Big\|\int_{Q}|k^{\phi}_{i,n}(x,y)|^{2}p_{Q} f(y) p_{Q} dy\Big\|^{\frac{1}{2}}_{L_{1}(\M)}dx \\&
\leq\Big(\int_{\R^{d}}\big\|p_{Q}\chi_{dist(x,Q)\thickapprox2^{-i}}\big\|_{L_{1}(\M)}dx\Big)
^{\frac{1}{2}}\Big(\int_{\R^{d}}\Big\|\int_{Q}|k^{\phi}_{i,n}(x,y)|^{2}p_{Q} f(y) p_{Q} dy\Big\|_{L_{1}(\M)}
dx\Big)^{\frac{1}{2}}.
\end{align*}
Since $i<n-1$ and $\ell(Q)=2^{-n}$, it is easy to verify that
\begin{align}\label{cz210}
\Big(\int_{\R^{d}}\big\|p_{Q}\chi_{dist(x,Q)\thickapprox2^{-i}}\big\|_{L_{1}(\M)}dx\Big)
^{\frac{1}{2}}\lesssim2^{-\frac{id}{2}}\big(\tau (p_{Q})\big)^{\frac{1}{2}}.
\end{align}
We then consider another term.  The Fubini theorem implies
$$\int_{\R^{d}}\Big\|\int_{Q}|k^{\phi}_{i,n}(x,y)|^{2}p_{Q} f(y) p_{Q} dy\Big\|_{L_{1}(\M)}
dx=\tau\int_{Q}\Big(\int_{\R^{d}}|k^{\phi}_{i,n}(x,y)|^{2}dx\Big)p_{Q} f(y) p_{Q} dy.$$
By Lemma \ref{HLX6}, this gives rise to
\begin{align}\label{cz214}
\begin{split}
&\quad\int_{\R^{d}}\Big\|\int_{Q}|k^{\phi}_{i,n}(x,y)|^{2}p_{Q} f(y) p_{Q} dy\Big\|_{L_{1}(\M)}
dx\\
 & \lesssim2^{id}(\delta^{2}_{2}(n-i)+\delta^{2}_{2}(n-i+1)+2^{2(i-n)})\varphi(fp_{Q} \chi_{Q}).
\end{split}
\end{align}
Finally, exploiting (\ref{cz210}) and (\ref{cz214}), we find
\begin{align*}
&\quad\int_{\R^{d}}\Big\|\int_{Q}|k^{\phi}_{i,n}(x,y)|^{2}p_{Q} f(y) p_{Q} dy\Big\|^{\frac{1}{2}}_{L_{\frac{1}{2}}(\M)}dx\\
 & \lesssim\big(\tau (p_{Q})\big)^{\frac{1}{2}}\Big((\delta^{2}_{2}(n-i)+\delta^{2}_{2}(n-i+1)+2^{2(i-n)})
\varphi(fp_{Q} \chi_{Q})\Big)^{\frac{1}{2}}\\& =\Big((\delta^{2}_{2}(n-i)+\delta^{2}_{2}(n-i+1)+2^{2(i-n)})\tau(p_{Q})
\varphi(fp_{Q} \chi_{Q})\Big)^{\frac{1}{2}}.
\end{align*}
This is the desired estimate (\ref{cz1009}). It is easy to see that the same argument works for estimate  \eqref{cz10092}.
\end{proof}

Now we are ready to show estimate \eqref{bd11}.
\begin{proof}[Proof of estimate \eqref{bd11}.]
By the assumption $m_{\lambda}(f)=0$, conclusion (iv) in Theorem \ref{czdecom} yields
$b_\mathit{off}=\sum_{n=1}^{\infty}b_{n}$, where $b_{n}=p_n(f-f_n)q_n+q_n(f-f_n)p_n$ as in (\ref{czd7}).
On the other hand, applying the similar argument in dealing with the diagonal term $b_{d}$ (see \eqref{bad34789}), we can deduce that $\zeta T_{\phi,i}b_{n}\zeta=0$ unless $i<n-1$.
Consequently, for $x\in\mathbb R^d$,
\begin{align*}
  \zeta(x)T^{\phi}_{j}b_{\mathit{off}}(x)\zeta(x)
  & =\zeta(x)\sum_{n=1}^{\infty}\sum_{i:i< j;i<n-1}\int_{\R^{d}}k^{\phi}_{i}(x,y)b_{n}(y)dy\zeta(x).
\end{align*}
Using the cancellation property of $b_{n}$ announced in Theorem \ref{czdecom} and the fact that $b_{n}$ is selfadjoint, we obtain
\begin{align*}
 \zeta(x)T^{\phi}_{j}b_{\mathit{off}}(x)\zeta(x)
  & =\zeta(x) \sum_{n=1}^{\infty}\sum_{i:i< j;i<n-1}\sum_{Q\in\Q_{n}}\int_{Q} k^{\phi}_{i,n}(x,y)
b_{n}(y) \, dy \zeta(x)\\
  & \leq\zeta(x) \sum_{n=1}^{\infty}\sum_{i:i<j;i<n-1}\sum_{Q\in\Q_{n}}\Big|\int_{Q} k^{\phi}_{i,n}(x,y)
b_{n}(y) \, dy \Big|\zeta(x)\\
&\leq\zeta(x) \sum_{n=1}^{\infty}\sum_{i:i<n-1}\sum_{Q\in\Q_{n}}\Big|\int_{Q} k^{\phi}_{i,n}(x,y)
b_{n}(y) \, dy \Big|\zeta(x)\\
&\triangleq \zeta(x)F_{2}(x)\zeta(x).
\end{align*}
Thus for all $j\in\Z$
\begin{align}\label{F451}
-\zeta(x)F_{2}(x)\zeta(x)\leq \zeta(x)T^{\phi}_{j}b_{\mathit{off}}(x)\zeta(x)\leq \zeta(x)F_{2}(x)\zeta(x).
\end{align}
Apparently, the proof of estimate (\ref{bd11}) would be completed if the following estimate were verified:
\begin{align}\label{F12111}
\|F_{2}\|_{1}\lesssim\|f\|_{1}.
\end{align}
Indeed,  set $e_{3,2}=\chi_{(0,\lambda]}(\zeta F_{2}\zeta)$. Then for any $j\in\Z$, estimates (\ref{F451}) and (\ref{F12111}) give
$$\|e_{3,2}\zeta T^{\phi}_{j}b_\mathit{off}\zeta e_{3,2}\|_{\infty}\leq\lambda\ \ \mbox{and}\ \
\varphi(e_{3,2}^{\perp})\leq\frac{\|F_{2}\|_{1}}{\lambda}\lesssim\frac{\| f\|_{1}}{\lambda},$$
which is the desired estimate (\ref{bd11}).

It remains to show (\ref{F12111}). By using the Fubini theorem, we clearly have
\begin{align*}
  \|F_{2}\|_{1}
  & =\tau \int_{\R^{d}}\sum_{n=1}^{\infty}\sum_{i:i<n-1}\sum_{Q\in\Q_{n}}\Big|\int_{Q} k^{\phi}_{i,n}(x,y)
b_{n}(y)dy \Big|dx\\
& = \sum_{n=1}^{\infty}\sum_{i:i<n-1}\sum_{Q\in\Q_{n}}\int_{\R^{d}}\Big\|\int_{Q} k^{\phi}_{i,n}(x,y)
b_{n}(y)dy\Big\|_{L_{1}(\M)}dx.
\end{align*}
The Minkowski inequality and the definition of $b_{n}$ imply that $\|F_{2}\|_{1}$ can be controlled by the sum of the following four terms
\begin{eqnarray*}
\|F_{2,1}\|_{1}& \triangleq &\sum_{n=1}^{\infty}\sum_{i:i<n-1}\sum_{Q\in\Q_{n}}\int_{\R^{d}}\Big\|\int_{Q} k^{\phi}_{i,n}(x,y)
p_{Q}f(y)q_{Q}dy\Big\|_{L_{1}(\M)}dx,\\
\|F_{2,2}\|_{1}& \triangleq &\sum_{n=1}^{\infty}\sum_{i:i<n-1}\sum_{Q\in\Q_{n}}\int_{\R^{d}}\Big\|\int_{Q} k^{\phi}_{i,n}(x,y)
q_{Q}f(y)p_{Q}dy\Big\|_{L_{1}(\M)}dx, \\
\|F_{2,3}\|_{1}& \triangleq &\sum_{n=1}^{\infty}\sum_{i:i<n-1}\sum_{Q\in\Q_{n}}\int_{\R^{d}}\Big\|\int_{Q} k^{\phi}_{i,n}(x,y)
p_{Q}f_{n}(y)q_{Q}dy\Big\|_{L_{1}(\M)}dx,\\
\|F_{2,4}\|_{1}& \triangleq &\sum_{n=1}^{\infty}\sum_{i:i<n-1}\sum_{Q\in\Q_{n}}\int_{\R^{d}}\Big\|\int_{Q} k^{\phi}_{i,n}(x,y)
q_{Q}f_{n}(y)p_{Q}dy\Big\|_{L_{1}(\M)}dx.
\end{eqnarray*}
Therefore, this leads to show
$$\max\Big\{\|F_{2,1}\|_{1},\|F_{2,2}\|_{1},\|F_{2,3}\|_{1},\|F_{2,4}\|_{1}\Big\}\lesssim\|f\|_{1}.$$

Consider $\|F_{2,1}\|_{1}$ firstly.  For convenience, we set
$$G_{i,n,Q}(x)\triangleq\int_{Q} k^{\phi}_{i,n}(x,y)
p_{Q}f(y)q_{Q}dy.$$
We now treat with $\|G_{i,n,Q}(x)\|_{L_{1}(\M)}$. 
We use Lemma \ref{mainlemma} to find that
\begin{align}\label{cz21412322}
\begin{split}
&\|G_{i,n,Q}(x)\|_{L_{1}(\M)}\\
&\leq\Big\|\Big(\int_{Q}|k^{\phi}_{i,n}(x,y)|^{2}p_{Q} f(y) p_{Q} dy\Big)^{\frac{1}{2}}\Big\|_{L_{1}(\M)}\Big\|\Big(\int_{Q}q_{Q} f(y) q_{Q} dy\Big)^{\frac{1}{2}}\Big\|_{L_{\infty}(\M)}\\
& = \Big\|\int_{Q}|k^{\phi}_{i,n}(x,y)|^{2}p_{Q} f(y) p_{Q} dy\Big\|^{\frac{1}{2}}_{L_{\frac{1}{2}}(\M)}\Big\|\Big(|Q|q_{Q} f_{Q} q_{Q} \Big)^{\frac{1}{2}}\Big\|_{L_{\infty}(\M)}.
\end{split}
\end{align}
Now applying the properties (\ref{czd5}) to the second term above, we get
\begin{align*}
\|G_{i,n,Q}(x)\|_{L_{1}(\M)}
\leq\Big\|\int_{Q}|k^{\phi}_{i,n}(x,y)|^{2}p_{Q} f(y) p_{Q} dy\Big\|^{\frac{1}{2}}_{L_{\frac{1}{2}}(\M)}(|Q|\lambda)^{\frac{1}{2}}.
\end{align*}
This provides us with the estimate
$$\|F_{2,1}\|_{1}\leq\sum_{n=1}^{\infty}\sum_{i:i<n-1}\sum_{
Q\in\Q_{n}}\int_{\R^{d}}\Big\|\int_{Q}|k^{\phi}_{i,n}(x,y)|^{2}p_{Q} f(y) p_{Q} dy\Big\|^{\frac{1}{2}}_{L_{\frac{1}{2}}(\M)}(|Q|\lambda)^{\frac{1}{2}}dx.$$
Now we apply (\ref{cz1009}) stated in  Lemma \ref{HLX13}, (\ref{6}) with $q=2$ and the Fubini theorem to deduce that
\begin{align*}
  \|F_{2,1}\|_{1}
& \lesssim\sum_{n=1}^{\infty}\sum_{i:i<n-1}\sum_{
Q\in\Q_{n}}(\delta_{2}(n-i)+\delta_{2}(n-i+1)+2^{i-n})\Big(\tau(p_{Q})
\varphi(fp_{Q} \chi_{Q})|Q|\lambda\Big)^{\frac{1}{2}}\\
& \leq\sum_{n=1}^{\infty}\sum_{
Q\in\Q_{n}}\sum_{m=1}^{\infty}(2\delta_{2}(m)+2^{-m})\Big(\tau(p_{Q})
\varphi(fp_{Q} \chi_{Q})|Q|\lambda\Big)^{\frac{1}{2}}\\
& \lesssim\sum_{n=1}^{\infty}\sum_{
Q\in\Q_{n}}\Big(\tau(p_{Q})
\varphi(fp_{Q} \chi_{Q})|Q|\lambda\Big)^{\frac{1}{2}}.
\end{align*}
Furthermore, the Cauchy-Schwarz inequality implies
\begin{align*}
  \sum_{
Q\in\Q_{n}}\Big(\tau(p_{Q})
\varphi(fp_{Q} \chi_{Q})|Q|\lambda\Big)^{\frac{1}{2}}
  & \leq\Big(\sum_{
Q\in\Q_{n}}\tau(p_{Q})
|Q|\lambda\Big)^{\frac{1}{2}}\Big(\sum_{
Q\in\Q_{n}}\varphi(fp_{Q} \chi_{Q})\Big)^{\frac{1}{2}}\\
& =\Big(\lambda\varphi(p_{n})\Big)^{\frac{1}{2}}
  \Big(\varphi(fp_{n})\Big)^{\frac{1}{2}}.
\end{align*}
Then applying the Cauchy-Schwarz inequality once more, we finally get
\begin{align*}
  \| F_{2,1}\|_{1}
  & \lesssim\sum_{n=1}^{\infty}\Big(\lambda\varphi(p_{n})\Big)
^{\frac{1}{2}}\Big(\varphi(fp_{n})\Big)^{\frac{1}{2}}\\
& \leq \Big(\sum_{n=1}^{\infty}\lambda\varphi(p_{n})\Big)^{\frac{1}{2}}
  \Big(\sum_{n=1}^{\infty}\varphi(fp_{n})\Big)^{\frac{1}{2}}\\
& =  \Big(\lambda\varphi(1-q)\Big)^{\frac{1}{2}}
  \Big(\varphi(f(1-q))\Big)^{\frac{1}{2}}\lesssim\|f\|_{1}.
\end{align*}
This finishes the estimate of $\| F_{2,1}\|_{1}$.

The same arguments work also for other three terms $\| F_{2,2}\|_{1}$, $\| F_{2,3}\|_{1}$ and $\| F_{2,4}\|_{1}$ by  \eqref{cz1009}, \eqref{cz10092} in Lemma \ref{HLX13} and noting $\varphi(f_{n}p_{n})=\varphi(fp_{n}).$ Thus we finish the proof.
\end{proof}
\begin{rk}\label{clever}\rm
We would like to point out that the estimate of $\|F_{2,1}\|_{1}$, in particular estimate (\ref{cz21412322}), is partially motivated by Cadilhac's note \cite{C2}.
\end{rk}
\begin{rk}\rm
It is worthy to point out that by showing the weak type $(1,1)$ of $(T_{j}^{\phi})_{j\in\Z}$, we see that Theorem \ref{p2} still holds true if we weaken the assumption of the strong type $(p_{0},p_{0})$ of $(T_{\varepsilon})_{\varepsilon>0}$ to the weak type $(p_{0},p_{0})$ of $(T_{\varepsilon})_{\varepsilon>0}$. The details are left to the interested readers.
\end{rk}


\section{Proof of Theorem \ref{t1}: maximal inequalities}

In this section, we prove the maximal inequalities announced in Theorem \ref{t1}. By Theorem \ref{p2}, the weak type $(1,1)$ result is a consequence of the strong type $(p,p)$ of $(T_\varepsilon )_{\varepsilon>0}$.  While to show the strong type $(p,p)$ of $(T_\varepsilon )_{\varepsilon>0}$, as mentioned in the introduction,
it suffices to show the following Cotlar inequality in terms of norm: for all $f\in L_p(\mathcal N)_+$ with $1<p<\infty$,
\begin{align}\label{cotlar norm}
\big\|{\sup_{\varepsilon>0}}^{+}T_{\varepsilon} f\big\|_p\lesssim  \big\|{\sup_{\varepsilon>0}}^{+}M_{\varepsilon} (Tf)\big\|_p+ \big\|{\sup_{\varepsilon>0}}^{+}M_{\varepsilon} f\big\|_p.
\end{align}
Indeed, by Mei's noncommutative Hardy-Littlewood maximal inequalities \cite{M} and the fact that $T$ is bounded on $L_{p}(\mathcal{N})$ under the kernel conditions (\ref{1}), (\ref{232}) and (\ref{con4}) (see e.g. \cite[Theorem A]{JP1}), \eqref{cotlar norm} implies
$$\big\|{\sup_{\varepsilon>0}}^{+}T_{\varepsilon} f\big\|_p\lesssim\|Tf\|_p+ \|f\|_p\lesssim\|f\|_p,\quad\forall\; f\in L_p(\mathcal N).$$


We need the following lemma, which should be well-known to experts, see e.g. \cite[Theorem 4.3]{CXY13}.
\begin{lem}\label{result of cxy}\rm
Let $\psi$ be a non-negative radial function on $\mathbb R^d$ such that $\psi(x)\lesssim\frac{1}{(1+|x|)^{d+\delta}}$ for $x\in\R^d$ with some $\delta>0$.
Let $\psi_\varepsilon(x)=\frac1{\varepsilon^d}\, \psi(\frac x\varepsilon)$ for $x\in\R^d$ and $\varepsilon>0$. Let $1<p\leq\infty$. Then
$$\big\|{\sup_{\varepsilon>0}}^+\psi_\varepsilon\ast f\big\|_p\lesssim\big\|{\sup_{\varepsilon>0}}^+M_\varepsilon f\big\|_p,\;\forall f\in L_p(\mathcal N).$$
\end{lem}
Now we are ready to prove the Cotlar inequality \eqref{cotlar norm}.

\begin{proof}[Proof of estimate \eqref{cotlar norm}.]
 Let $\varphi$ be a smooth, radial, radially decreasing non-negative
function which is supported in the ball $B(0,\frac{1}{2})$ and $\int_{\R^{d}}\varphi(x)dx=1$. We will use the notation $\varphi_\varepsilon(x)=\frac1{\varepsilon^d}\, \varphi(\frac x\varepsilon)$ for $x\in\R^d$. 
Fix $f\in L_p(\mathcal N)_+$, now we express $T_{\varepsilon}f$ as
\begin{align}\label{k1}
T_{\varepsilon}f=\varphi_{\varepsilon}\ast(Tf)+
(T_{\varepsilon}f-\varphi_{\varepsilon}\ast(Tf)).
\end{align}

By Lemma \ref{result of cxy}, we trivially have
\begin{align*}
\big\|{\sup_{\varepsilon>0}}^+\varphi_{\varepsilon}\ast(Tf)\big\|_p\lesssim \big\|{\sup_{\varepsilon>0}}^{+}M_{\varepsilon} (Tf)\big\|_p.
\end{align*}

On the other hand, we rewrite the difference in \eqref{k1} as
$$T_{\varepsilon}f(x)-\varphi_{\varepsilon}\ast(Tf)=(k\chi_{|\cdot|>\varepsilon}-\varphi_\varepsilon\ast k)\ast f.$$
It is easy to verify that (see e.g. \cite[Theorem 5.3.4]{Gra2008} )
\begin{align*}
|k\chi_{|\cdot|>\varepsilon}-\varphi_\varepsilon\ast k|\lesssim\psi_{\varepsilon},
\end{align*}
where
$$\psi(x)=\frac{1}{(1+|x|)^{d+\gamma}}$$
with $\gamma$ being the regularity index of the kernel.
By decomposing $k\chi_{|\cdot|>\varepsilon}-\varphi_\varepsilon\ast k$ into a linear combination of four positive functions, one gets
\begin{align*}
\big\|{\sup_{\varepsilon>0}}^+(T_{\varepsilon}f-\varphi_{\varepsilon}\ast(Tf))\big\|_p&\lesssim \big\|{\sup_{\varepsilon>0}}^+|k\chi_{|\cdot|>\varepsilon}-\varphi_\varepsilon\ast k|\ast f\big\|_p\\
&\lesssim \big\|{\sup_{\varepsilon>0}}^+\psi_{\varepsilon}\ast f\big\|_p\lesssim \big\|{\sup_{\varepsilon>0}}^+M_\varepsilon f\big\|_p.
\end{align*}

Combining the above estimates, we get the desired Cotlar inequality \eqref{cotlar norm}.
\end{proof}

\section{Proof of Theorem \ref{t3}: maximal inequalities}
In this section, we prove the maximal inequalities stated in Theorem \ref{t3}. We start with the strong type $(p,p)$ estimate of $(T_{\Omega,\varepsilon})_{\varepsilon>0}$.
\subsection{Strong type $(p,p)$ estimate of $(T_{\Omega,\varepsilon})_{\varepsilon>0}$}
For convenience, we denote $k_{\Omega}(x)\triangleq{\Omega(x)}/{|x|^{d}}$. Without loss of generality, we may assume that $\Omega$ (and thus  $k_{\Omega}$) is real-valued. Since $\Omega$ on $S^{d-1}$ can be decomposed into its even and odd parts,
\begin{equation}\label{rough54}
\Omega_{e}(x)=\frac{1}{2}(\Omega(x)+\Omega(-x)),\ \ \Omega_{o}(x)=\frac{1}{2}(\Omega(x)-\Omega(-x)),
\end{equation}
it suffices to consider $\Omega$ with the cases of being odd and even, respectively.
We begin with the case of being odd.
\subsubsection{Case: $\Omega$ is odd}
As in the commutative case, the method of rotation will play a crucial role in the study of $T_{\Omega}$ and $(T_{\Omega,\varepsilon})_{\varepsilon>0}$. We first study the directional Hilbert transforms. Given a unit vector $\theta$ in $\R^{d}$, the directional Hilbert transform in the direction $\theta$ is defined as: for $f\in C^\infty_c(\mathbb R^d)\otimes \mathcal S_\mathcal M$,
\begin{equation}\label{rough1}
H_{\theta}f(x)=p.v.\frac{1}{\pi}\int_{\R}f(x-t\theta)\frac{dt}{t}.
\end{equation}
Likewise, we define the associated directional truncated Hilbert transform for $\varepsilon>0$,
\begin{equation}\label{rough2}
H_{\theta,\varepsilon}f(x)=\frac{1}{\pi}\int_{|t|>\varepsilon}f(x-t\theta)\frac{dt}{t}.
\end{equation}
\begin{lem}\label{rough4}\rm
Let $H_{\theta}$ and $H_{\theta,\varepsilon}$ be defined as in (\ref{rough1}) and (\ref{rough2}), respectively. Then for all $f\in L_{p}(\mathcal{N})$ with $1<p<\infty$
$$\big\|H_{\theta}f\big\|_{p}\lesssim\|f\|_{p}\ \ \ \mbox{and}\ \ \ \|{\sup_{\varepsilon>0}}^{+}H_{\theta,\varepsilon}f\|_{p}\lesssim\|f\|_{p}.$$
\end{lem}

The above result should be known to experts, and for the sake of completeness, we give a sketch of the proof.

\begin{proof}
Let $e_{1}=(1,0,\dotsm,0)$ be the unit vector. Then it is easy to see that the operator $H_{e_{1}}$ can be written as $H\otimes id_{L_\infty(\mathbb R^{d-1})}$ the tensor product of the usual Hilbert transform with identity map on functions on $\mathbb R^{d-1}$. Thus, $H_{e_{1}}$ is bounded on $L_{p}(\mathcal{N})$ with norm equal to the completely bounded norm of the usual Hilbert transform on $L_p$ (see e.g. \cite[Theorem A]{JP1}). Observing that for any orthogonal matrix $A$, the following identity holds
\begin{equation}\label{rough}
H_{A(e_{1})}f(x)=H_{e_{1}}(f\circ A)(A^{-1}x).
\end{equation}
This implies that the $L_{p}$ boundedness of $H_{\theta}$ can be reduced to that of $H_{e_{1}}$, which proves the first inequality.

Next, we consider the second inequality. Clearly, identity (\ref{rough}) is also valid for $H_{\theta,\varepsilon}$. Consequently, it suffices to show that $(H_{e_1,\varepsilon})_{\varepsilon>0}$ is of strong type $(p,p)$.
Let $f\in L_p(\mathcal{N})$. Without loss of generality, we may assume that $f$ is positive. Fixing $x_2,...,x_d\in\R$, we consider $f(\cdot,x_2,...,x_d)$ as a function in $L_p(L_\infty(\R)\overline{\otimes}\M)_{+}$. By Theorem \ref{t1}, we know that for $1<p<\infty$
$$\big\|{\sup_{\varepsilon>0}}^+H_{\varepsilon}f(\cdot,x_{2},...,x_{d})\big\|_{L_{p}(L_\infty(\R)\overline{\otimes}\M)}
\lesssim\|f(\cdot,x_{2},...,x_{d})\|_{L_{p}(L_\infty(\R)\overline{\otimes}\M)}.$$
This implies that, by Remark \ref{rk:MaxFunct}, there exists a positive function $F(\cdot,x_2,\cdots,x_d)\in L_p(L_\infty(\R)\overline{\otimes}\M)$ such that for any $\varepsilon>0$
\begin{eqnarray*}
-F(x)\leq(H_\varepsilon\otimes id_{L_\infty(\mathbb R^{d-1})})f(x)= H_{e_1,\varepsilon}f(x)\leq F(x),
\end{eqnarray*}
$$\\
\tau\int_{\R}|F(x_1,x_2,\cdots,x_d)|^{p}dx_1  \lesssim \tau\int_{\R}|f(x_1,x_2,\cdots,x_d)|^{p}dx_1.$$
Then it is easy to see that
$$\|F\|_{L_p(\mathcal{N})}\lesssim\|f\|_{L_p(\mathcal{N})}.$$
Therefore, we conclude that $(H_{e_1,\varepsilon})_{\varepsilon>0}$ is of strong type $(p,p)$.
\end{proof}

\begin{prop}\label{rough18}\rm
If $\Omega$ is odd and integrable over $S^{d-1}$, then $(T_{\Omega,\varepsilon})_{\varepsilon>0}$ is of strong type $(p,p)$ for all $1<p<\infty$.
\end{prop}

\begin{proof}
By switching to polar coordinates and the fact that $\Omega$ is odd, we get the following identities:
\begin{align*}
\int_{|y|>\varepsilon}k_{\Omega}(y)f(x-y)dy & =\int_{S^{d-1}}\Omega(\theta)\int_{\varepsilon}^{\infty}f(x-r\theta)\frac{dr}{r}d\theta\\& =-\int_{S^{d-1}}\Omega(\theta)\int_{\varepsilon}^{\infty}f(x+r\theta)\frac{dr}{r}d\theta,
\end{align*}
where the second equality is a consequence of the first one via the change of variables $\theta\rightarrow-\theta$ based on the fact that $\Omega$ is an odd function . Thus we obtain
\begin{align*}
\int_{|y|>\varepsilon}k_{\Omega}(y)f(x-y)dy & =\frac{1}{2}\int_{S^{d-1}}\Omega(\theta)\int_{\varepsilon}^{\infty}\frac{f(x-r\theta)-f(x+r\theta)}{r}
drd\theta\\& =\frac{\pi}{2}\int_{S^{d-1}}\Omega(\theta)H_{\theta,\varepsilon}f(x)d\theta.
\end{align*}
Now by the triangle inequality and  Lemma \ref{rough4}, we get that
$$\Big\|{\sup_{\varepsilon>0}}^{+}T_{\Omega,\varepsilon}f\Big\|_{p}\lesssim
\int_{S^{d-1}}|\Omega(\theta)|\Big\|{\sup_{\varepsilon>0}}^{+}H_{\theta,\varepsilon}f\Big\|_{p}d\theta
\lesssim\|f\|_{p}.$$
Therefore, we
obtain the announced result.
\end{proof}
\subsubsection{Case: $\Omega$ is even}
In this case, we need the directional Hardy-Littlewood maximal inequality.
\begin{lem}\label{rough13}\rm
Let $1<p\leq\infty$ and $\theta$ be a unit vector in $\mathbb R^d$. Then for any $f\in L_{p}(\mathcal{N})$,
$$\Big\|{\sup_{\varepsilon>0}}^{+}f_{\theta,\varepsilon}\Big\|_{p}\lesssim\|f\|_{p},$$
where
\begin{equation}\label{rough12}
f_{\theta,\varepsilon}(x)=\frac{1}{2\varepsilon}\int_{|r|\leq \varepsilon}f(x-r\theta)dr.
\end{equation}
\end{lem}
This result can be obtained by the same argument as that for directional Hilbert transform, which should be also known to experts. We omit the details.

\begin{prop}\label{rough89}\rm
Let $\Omega$ be an even function on $S^{d-1}$ with mean value zero and $\Omega\in L\log^{+}L(S^{d-1})$. Then for any $f\in L_{p}(\mathcal{N})$ with $1<p<\infty$,
$$\Big\|{\sup_{\varepsilon>0}}^{+}T_{\Omega,\varepsilon}f\Big\|_{p}\lesssim\|f\|_{p}.$$
\end{prop}

\begin{proof}
Let $\varphi$ a smooth radial function such that $\varphi(x)=0$ for $|x|\leq\frac{1}{4}$, $\varphi(x)=1$ for $|x|\geq\frac{3}{4}$, and $0\leq \varphi(x)\leq1$ for all $x\in\R^{d}$. Fix $f\in L_{p}(\mathcal{N})$; without loss of generality, we may assume that $f$ is positive. We define the smoothly truncated singular integral as
$$\widetilde{T}_{\Omega,\varepsilon}f(x)=\int_{\R^{d}}k_{\Omega}(y)\varphi\Big(\frac{y}{\varepsilon}\Big)
f(x-y)dy.$$
By the support of $\varphi$, we deduce that
\begin{align*}
\widetilde{T}_{\Omega,\varepsilon}f(x)-T_{\Omega,\varepsilon}f(x) & =\int_{\frac{\varepsilon}{4}\leq|y|\leq\frac{3\varepsilon}{4}}k_{\Omega}(y)\varphi\Big(\frac{y}{\varepsilon}\Big)
f(x-y)dy\\& \leq\int_{S^{d-1}}|\Omega(\theta)|\big[\frac{4}{\varepsilon}
\int_{\frac{\varepsilon}{4}}^{\frac{3}{4}\varepsilon}f(x-r\theta)dr\big]d\theta\\
&\lesssim \int_{S^{d-1}}|\Omega(\theta)|f_{\theta,\varepsilon}(x)d\theta.
\end{align*}
On the other hand, we have
$$-\int_{S^{d-1}}|\Omega(\theta)|f_{\theta,\varepsilon}(x)d\theta\lesssim
\widetilde{T}_{\Omega,\varepsilon}f(x)-T_{\Omega,\varepsilon}f(x).$$
Thus by Lemma \ref{rough13}, one gets
$$\Big\|{\sup_{\varepsilon>0}}^{+}(\widetilde{T}_{\Omega,\varepsilon}f-
T_{\Omega,\varepsilon}f)\Big\|_{p}\lesssim\|\Omega\|_{1}\big\|{\sup_{\varepsilon>0}}^{+}
f_{\theta,\varepsilon}\big\|_{p}\lesssim\|f\|_{p}.$$
Thus, to finish the proof, by rewriting
$$T_{\Omega,\varepsilon}f(x)=T_{\Omega,\varepsilon}f(x)-\widetilde{T}_{\Omega,\varepsilon}f(x)+\widetilde{T}_{\Omega,\varepsilon}f(x),$$
it suffices to consider the required strong type $(p,p)$ estimate for the smoothly truncated singular operator $(\widetilde{T}_{\Omega,\varepsilon})_{\varepsilon>0}$.

The identity property of the Riesz transforms $-\sum_{j=1}^dR_j^2=I$ will play a key role. Let $s_{j}$ (resp. $k_{j}$) be the $j$-th Riesz transform of $k_{\Omega}\varphi$ (resp. $k_{\Omega}$).
The identity property implies
$$\widetilde{T}_{\Omega,\varepsilon}f(x)=\int_{\R^{d}}\frac{1}{\varepsilon^{d}}
k_{\Omega}\Big(\frac{y}{\varepsilon}\Big)\varphi\Big(\frac{y}{\varepsilon}\Big)f(x-y)dy=-\Big(\sum_{j=1}^{d}
\frac{1}{\varepsilon^{d}}s_{j}\Big(\frac{\cdot}{\varepsilon}\Big)\ast R_{j}f\Big)(x),$$
where in the second equality we also used 
the homogeneity of $R_j$.
Therefore, we are able to decompose $\widetilde{T}_{\Omega,\varepsilon}f(x)$ as
$$-\widetilde{T}_{\Omega,\varepsilon}f(x)=\sum_{j=1}^{d}\int_{\R^{d}}
\frac{1}{\varepsilon^{d}}s_{j}\Big(\frac{x-y}{\varepsilon}\Big)R_{j}f(y)dy\triangleq N_{1,\varepsilon}f(x)+N_{2,\varepsilon}f(x)+N_{3,\varepsilon}f(x),$$
where
\begin{eqnarray*}
N_{1,\varepsilon}f(x)& = &\sum_{j=1}^{d}\frac{1}{\varepsilon^{d}}\int_{|x-y|\leq\varepsilon}s_{j}\Big(\frac{x-y}{\varepsilon}\Big) R_{j}f(y)dy,\\
N_{2,\varepsilon}f(x) & = & \sum_{j=1}^{d}\frac{1}{\varepsilon^{d}}\int_{|x-y|>\varepsilon}\Big[
s_{j}\Big(\frac{x-y}{\varepsilon}\Big)-k_{j}\Big(\frac{x-y}{\varepsilon}\Big)\Big]R_{j}f(y)dy, \\
N_{3,\varepsilon}f(x) & = & \sum_{j=1}^{d}\frac{1}{\varepsilon^{d}}\int_{|x-y|>\varepsilon}k_{j}\Big(\frac{x-y}{\varepsilon}\Big) R_{j}f(y)dy.
\end{eqnarray*}

We first deal with $N_{2,\varepsilon}f$. It follows from the definitions of $s_{j}$ and $k_{j}$ that
\begin{align*}
s_{j}(z)-k_{j}(z) & =c_{d}\lim_{\varepsilon\rightarrow0}\int_{|y|\geq\varepsilon}k_{\Omega}(y)\big(\varphi(y)-1\big)
\frac{z_{j}-y_{j}}{|z-y|^{d+1}}dy\\
&= c_{d}\int_{|y|\leq\frac{3}{4}}k_{\Omega}(y)\big(\varphi(y)-1\big)
\Big\{\frac{z_{j}-y_{j}}{|z-y|^{d+1}}-\frac{z_{j}}{|z|^{d+1}}\Big\}dy
\end{align*}
whenever $|z|\geq1$. By using the mean value theorem, we see that the last expression above can be controlled by
$$c_{d}\int_{|y|\leq\frac{3}{4}}|k_{\Omega}(y)|\frac{|y|}{|z|^{d+1}}dy=
c'_{d}\|\Omega\|_{1}|z|^{-(d+1)},$$
whenever $|z|\geq1$. By applying this estimate, we get that the j-th term in $N_{2,\varepsilon}f(x)$ is bounded by
$$c'_{d}\frac{\|\Omega\|_{1}}{\varepsilon^{d}}\int_{|x-y|>\varepsilon}\frac{|R_{j}f(y)|dy}
{\big(\frac{|x-y|}{\varepsilon}\big)^{d+1}}\leq c'_{d}\frac{2\|\Omega\|_{1}}{2^{-d}\varepsilon^{d}}\int_{\R^{d}}\frac{|R_{j}f(y)|dy}
{\big(1+\frac{|x-y|}{\varepsilon}\big)^{d+1}}.$$
Clearly, the j-th term in $N_{2,\varepsilon}f(x)$ also has the lower bound
$$-c'_{d}\frac{2\|\Omega\|_{1}}{2^{-d}\varepsilon^{d}}\int_{\R^{d}}\frac{|R_{j}f(y)|dy}
{\big(1+\frac{|x-y|}{\varepsilon}\big)^{d+1}}.$$
Applying Lemma \ref{result of cxy} and the fact that the Riesz transform $R_j$ is bounded on $L_p(\mathcal{N})$, we get
\begin{equation}\label{rough16}
\|{\sup_{\varepsilon>0}}^{+}N_{2,\varepsilon}f\|_{p}\lesssim
\sum_{j=1}^d\||R_{j}f|\|_{p}=\sum_{j=1}^d\|R_{j}f\|_{p}\lesssim\|f\|_{p}.
\end{equation}

Next we consider the term $N_{3,\varepsilon}f$. It is already shown in \cite[Theorem 5.2.10]{Gra2008} that for any $1\leq j\leq d$
$$k_{j}(x)={\Omega_{j}(x)}/{|x|^{d}},$$
where $\Omega_{j}$ is some homogeneous integrable odd function on $S^{d-1}$ satisfying
$$\|\Omega_{j}\|_{1}\leq c_{d}\Big(\int_{S^{d-1}}|\Omega(\theta)|\log^{+}|\Omega(\theta)|d\theta+1\Big),$$
where $\log^{+}s=\log s$ if $s\geq1$ and $\log^{+}s=0$ if $0\leq s<1$. Consequently,
$$N_{3,\varepsilon}f(x)= \sum_{j=1}^{d}\int_{|x-y|>\varepsilon}\frac{\Omega_{j}(x-y)}{|x-y|^{d}}R_{j}f(y)dy.$$
Now we use Theorem \ref{rough18} and the fact that the Riesz transform is bounded on $L_p(\mathcal{N})$ to conclude that
\begin{equation}\label{rough17}
\Big\|{\sup_{\varepsilon>0}}^{+}N_{3,\varepsilon}f\Big\|_{p}\lesssim
\sum_{j=1}^d\|R_{j}f\|_{p}\lesssim\|f\|_{p}.
\end{equation}

Finally, we turn our attention to the term $N_{1,\varepsilon}f$. By \cite[Theorem 5.2.10]{Gra2008}, for each $1\leq j\leq d$, there exists a non-negative homogeneous of degree zero function $g_{j}$ on $\R^{d}$ satisfying
\begin{equation}\label{rough168977}
|s_{j}(x)|\leq g_{j}(x)\ \ \ \mbox{when}\ |x|\leq1
\end{equation}
and
\begin{equation}\label{rough1677}
\int_{S^{d-1}}|g_{j}(\theta)|d\theta\leq c_{d}\Big(\int_{S^{d-1}}|\Omega(\theta)|\log^{+}|\Omega(\theta)|d\theta+1\Big).
\end{equation}
Therefore, we get
\begin{align*}
N_{1,\varepsilon}f(x) & \lesssim \sum_{j=1}^{d}\frac{1}{\varepsilon^{d}}\int_{|z|\leq\varepsilon}|s_{j}(z)||R_{j}f(x-z)|dz\\
&=\sum_{j=1}^{d}\frac{1}{\varepsilon^{d}}\int_{0}^{\varepsilon}\int_{S^{d-1}}|s_{j}(r\theta)|
|R_{j}f(x-r\theta)|r^{d-1}d\theta dr\\
&\leq \sum_{j=1}^{d}\int_{S^{d-1}}|g_{j}(\theta)|\frac{1}{\varepsilon^{d}}
\int_{0}^{\varepsilon}|R_{j}f(x-r\theta)|r^{d-1}drd\theta\\
&\leq \sum_{j=1}^{d}\int_{S^{d-1}}|g_{j}(\theta)|\frac{1}{\varepsilon}
\int_{0}^{\varepsilon}|R_{j}f(x-r\theta)|drd\theta.
\end{align*}
where the second inequality follows from the homogeneity of $s_{j}$ and (\ref{rough168977}). On the other hand, it is easy to verify that
$$-\sum_{j=1}^{d}\int_{S^{d-1}}|g_{j}(\theta)|\frac{1}{\varepsilon}
\int_{0}^{\varepsilon}|R_{j}f(x-r\theta)|drd\theta\lesssim N_{1,\varepsilon}f(x).$$
Hence by Lemma \ref{rough13}, (\ref{rough1677}) and the $L_p$ boundedness of $R_j$, we arrive at
\begin{equation}\label{rough19}
\|{\sup_{\varepsilon>0}}^{+}N_{1,\varepsilon}f\|_{p}\lesssim
\sum_{j=1}^d\||R_{j}f|\|_{p}=\sum_{j=1}^d\|R_{j}f\|_{p}\lesssim\|f\|_{p}.
\end{equation}
Finally, combining (\ref{rough16}), (\ref{rough17}) and (\ref{rough19}), we get the desired strong type $(p,p)$ estimate of $(\widetilde{T}_{\Omega,\varepsilon})_{\varepsilon>0}$. This completes the proof.
\end{proof}

Combining Proposition \ref{rough18} and Proposition \ref{rough89}, we get Theorem \ref{t3} (i).

\subsection{Weak type $(1,1)$ estimate of $(T_{\Omega,\varepsilon})_{\varepsilon>0}$}
\begin{lem}\label{rough8998}\rm
The kernel $k_\Omega$ given in Theorem \ref{t3} satisfies the smoothness condition (\ref{6}) with $q=2$. More precisely,
$$\sum_{m=1}^\infty\delta_2(m)\lc\int_0^1\frac{\omega_2(s)}{s}ds+\|\Omega\|_{L_2(S^{d-1})}.$$
\end{lem}
\begin{proof}
Let $R>0$, $|u|\leq R$ and $m\in\mathbb N_{+}$. Then
\begin{equation}
\begin{split}
&\int_{2^mR\leq|x|\leq2^{m+1}R}|k_\Omega(x+u)-k_{\Omega}(x)|^2dx\\
&\lc\int_{2^mR\leq|x|\leq2^{m+1}R}|\Omega(x+u)-\Omega(x)|^2|x+u|^{-2d}dx\\
&\quad+\int_{2^mR\leq|x|\leq2^{m+1}R}|\Omega(x)|^2||x+u|^{-2d}-|x|^{-2d}|dx\\
&\triangleq J_1(m,R)+J_2(m,R).
\end{split}\end{equation}

We first consider $J_1(m,R)$. Since $|u|\leq R$ and $2^mR\leq|x|\leq2^{m+1}R$, we see that $2^{m-1}R\leq |x+u|\leq 2^{m+2}R$. By a change of variable $x=r\tet$, then using the fact that the kernel $\Omega$ is of zero homogeneity, we get
\begin{equation*}
J_1(m,R)\lesssim(2^mR)^{-2d}\int_{2^mR}^{2^{m+1}R}r^{d-1}\int_{S^{d-1}}|\Omega(\tet+u/r)-\Omega(\tet)|^2d\tet dr\lc \omega_{2}^2(2^{-m})(2^mR)^{-d}.
\end{equation*}

For the second term $J_2(m,R)$, using the mean value theorem, it is not difficult to get that
\begin{equation*}
J_2(m,R)\lc \|\Omega\|^2_{L_2(S^{d-1})}2^{-m}(2^mR)^{-d}.
\end{equation*}

Hence combining the estimates of $J_1(m,R)$ and $J_2(m,R)$, we see that
$
\delta_2(m)\lc\omega_{2}(2^{-m})+2^{-m/2}\|\Omega\|_{L_2(S^{d-1})}.
$
Therefore we obtain that
\begin{equation*}
\sum_{m=1}^\infty\delta_2(m)\lc\int_0^1\frac{\omega_2(s)}{s}ds+\|\Omega\|_{L_2(S^{d-1})},
\end{equation*}
where we used the fact that $\omega_2(s)$ is monotone increasing.
\end{proof}

Finally, we are ready to show the weak type $(1,1)$ estimate of $(T_{\Omega,\varepsilon})_{\varepsilon>0}$.

\begin{proof}[Proof of Theorem \ref{t3} (ii).]
Lemma \ref{rough8998} ensures that $k_{\Omega}$ satisfies the required regularity condition needed in Theorem \ref{p2}. As a consequence, we can finish the proof of the weak type $(1,1)$ estimate in the similar way as in the proof of Theorem \ref{p2}. The only difference is that when reducing to the lacunary subsequence in \eqref{max191}, Mei's Hardy-Littlewood maximal inequalities should be replaced by the generalized ones for rough kernels
\begin{align}\label{Lai}
\|(M_{\Omega,r}f)_{r>0}\|_{\Lambda_{1,\infty}(\mathcal N;\ell_\infty)}\lesssim \|f\|_1,\;\|(M_{\Omega,r}f)_{r>0}\|_{L_{p}(\mathcal N;\ell_\infty)}\lesssim \|f\|_p
\end{align}
for $1<p\leq\infty$, where
$$M_{\Omega,r}f(x)=\frac{1}{r^d}\int_{|x-y|\leq r}\Omega(x-y)f(y)dy.$$
These maximal inequalities \eqref{Lai} have been established by the second author in \cite{Lai} for the rough kernels which include the $L_{2}$-Dini assumption in Theorem \ref{t3} (ii). Thus using the same argument as for Theorem \ref{p2}, we complete the proof of Theorem \ref{t3} (ii).
\end{proof}

\section{Proof of pointwise convergence results}

In this section, we deal with the noncommutative pointwise convergence results stated in
Theorem \ref{t1} and Theorem \ref{t3}. To this end, we recall first the definition of bilaterally almost uniform convergence which can be viewed as the noncommutative analogue of almost everywhere convergence. The following definition of almost uniform convergence was introduced by Lance \cite{Lan76}.
\begin{defi}\label{3456}\rm
Let $\mathcal {M}$ be a von Neumann algebra equipped with a \emph{n.s.f} trace $\tau$. Let $x_k,x\in L_0(\M)$.
\begin{itemize}
\item[(i)]  $(x_k)$ is said to converge bilaterally almost uniformly (b.a.u. in short) to $x$ if for any $\delta>0$, there is a
projection $e\in \M$ such that
$$\tau(e^\bot)<\delta \ \ \  \mbox{and} \ \ \ \lim _{k\rightarrow\infty}\|e(x_k-x)e\|_\infty=0.$$

\item[(ii)]  $(x_k)$ is said to converge almost uniformly (a.u. in short) to $x$ if for any $\delta>0$, there is a projection $e\in \M$
such that
$$\tau(e^\bot)<\delta \ \ \  \mbox{and} \ \ \ \lim _{k\rightarrow\infty}\|(x_k-x)e\|_\infty=0.$$
\end{itemize}
\end{defi}
Obviously, $x_{k}\xrightarrow{\rm a.u}x$ implies $x_{k}\xrightarrow{\rm b.a.u}x$. Note that in the commutative case on probability space, both convergences in Definition \ref{3456}
are equivalent to the usual almost everywhere convergence in terms of Egorov's theorem.

In the following, we prove conclusion (iii) in Theorem \ref{t1}, while the treatment of conclusion (iii) stated in Theorem \ref{t3} is similar.


\begin{proof}[Proof of Theorem \ref{t1} (iii).]
Let $(\varepsilon_j)_{j\in\mathbb N}$ be the subsequence appeared in \eqref{can}. Let $f\in L_{p}(\mathcal{N})$. It is already known from the uniform boundedness principle that $T_{\varepsilon_j}f\rightarrow Tf$ in $L_p(\mathcal N)$ (in $L_{1,\infty}(\mathcal N)$ if $p=1$). Thus we only need to show the b.a.u convergence of $(T_{\varepsilon_j}f)_j$ as $j\rightarrow\infty$, since then the limit must be $Tf$ (see e.g. \cite{CXY13}). Then it is known that equivalently (see e.g. \cite[Proposition 1.3]{CL1}) it suffices to show that for any $\delta>0$, there exists a projection $e\in\mathcal N$ such that
\begin{align}\label{key est}
\varphi(e^{\perp})<\delta\ \mbox{and}\ \|e(T_{\varepsilon_k}f-T_{\varepsilon_\ell}f)e\|_{\infty}\rightarrow0\ \mbox{as}\ k,\ell\rightarrow\infty.
\end{align}

By the density of $C^\infty_c(\mathbb R^d)\otimes S_{\mathcal M}$ in $L_p(\mathcal N)$, for each $\varepsilon>0$, there exists $g$ of the form
$\sum_i\varphi_{i}\otimes m_{i}$ with finite sum of $i$, $\varphi_{i}\in C^\infty_c(\mathbb R^d)$ and $m_{i}\in S_{\mathcal M}$ such that
\begin{equation}\label{aucon3456}
\|f-g\|_{p}\leq\varepsilon.
\end{equation}
Fix $\delta>0$. For each $n\geq1$, by \eqref{aucon3456}, there exists $g_n$ such that $\|f-g_n\|_p^p\leq\frac{\delta}{2^n n^p}$. Next applying that $(T_{\varepsilon_{j}})_{j\in\N}$ is of weak type $(p,p)$ ($1\leq p<\infty$), there exists a projection $e_{n}\in \mathcal{N}$ such that
\begin{equation*}\sup_{j}\|e_{n}T_{\varepsilon_{j}} (f-g_{n})e_{n}\|_{\infty}<\frac1n\ \mbox{and}\
\varphi(e^\perp_{n})<n^p\|f-g_n\|_p^p\leq\frac{\delta}{2^{n}}.
\end{equation*}
Let $e=\wedge_{n}e_{n}$. Then we have
\begin{align}\label{err}\varphi(e^\perp)<\delta\ \mbox{and}\ \sup_{j}\|eT_{\varepsilon_{j}} (f-g_{n})e\|_{\infty}<\frac{1}{n},\ \forall\ n\geq1.\end{align}

On the other hand, we claim that
\begin{align}\label{uni}
\lim_{\ell,k\rightarrow+\infty}\|T_{\varepsilon_{\ell}}g_{n}-T_{\varepsilon_{k}}g_{n}\|_{\infty}=0,\ \forall\ n\geq1.
\end{align}
Indeed, for $g_n=\sum_i\varphi_{i}\otimes m_{i}$ with finite sum of $i$ and $\varepsilon_\ell<\varepsilon_k$, one has
\begin{align*}
T_{\varepsilon_{\ell}}g_{n}(x)-T_{\varepsilon_{k}}g_{n}(x) & = \sum_{i}\int_{\varepsilon_{\ell}<|y|\leq\varepsilon_{k}}k(y)\varphi_{i}(x-y)dy\otimes m_{i}\\
&=\sum_{i}\int_{\varepsilon_{\ell}<|y|\leq\varepsilon_{k}}k(y)\big(\varphi_{i}(x-y)-\varphi_{i}(x)\big)dy\otimes m_{i}\\
&\ \ \ + \sum_{i}\int_{\varepsilon_{\ell}<|y|\leq\varepsilon_{k}}k(y)\varphi_{i}(x)dy\otimes m_{i},
\end{align*}
which tends to 0 as $\ell,k\rightarrow+\infty$ by simple calculation using the size/cancellation condition (\ref{1})/ (\ref{can}) of the kernel.

The uniform convergence \eqref{uni} implies that for any $n\geq1$, there exists a positive constant $N_{n}$ such that for any $\ell,k>N_{n}$
$$\|T_{\varepsilon_{\ell}}g_{n}-T_{\varepsilon_{k}}g_{n}\|_{\infty}<\frac{1}{n}.$$
Then together with \eqref{err}, for any $\ell,k>N_{n}$,
\begin{align*}
\|e(T_{\varepsilon_{\ell}}f-T_{\varepsilon_{k}}f)e\|_{\infty} &\leq \|T_{\varepsilon_{\ell}}g_{n}-T_{\varepsilon_{k}}g_{n}\|_{\infty}+\|e
T_{\varepsilon_{\ell}}(f-g_{n})e\|_{\infty}
+\|eT_{\varepsilon_{k}}(f-g_n)e\|_{\infty}\\
&<\frac{3}{n},
\end{align*}
which yields
$$\lim_{\ell,k\rightarrow\infty}\|e(T_{\varepsilon_{\ell}}f-T_{\varepsilon_{k}}f)e\|_{\infty}=0.$$
This completes the proof.
\end{proof}

\begin{rk}\rm
It is worthy to point out that at the moment of writing we have no idea how to strengthen the b.a.u. convergence in Theorem \ref{t1} and Theorem \ref{t3} to a.u. convergence. One reason is that Calder\'on-Zygmund operators are not positive.
\end{rk}

\noindent {\bf Acknowledgements} \
We are very grateful to L\'{e}onard Cadilhac and \'{E}ric Ricard  for communicating to us the present version of Calder\'on-Zygmund decomposition as in Theorem \ref{czdecom}. In particular, we thank L\'{e}onard Cadilhac for providing us with the note with a proof of the weak type $(1,1)$ estimate of standard Calder\'on-Zygmund operator \cite{C2}. The first author would also like to thank Javier Parcet for the discussion around (\ref{cotlar norm}) when he worked in Madrid as a post-doctor.


\end{document}